\documentclass{amsart}
\usepackage[all]{xy}
\newtheorem{theorem}{Theorem}[section]
\newtheorem{lemma}[theorem]{Lemma}
\newtheorem{corollary}[theorem]{Corollary}
\newtheorem{proposition}[theorem]{Proposition}
\newtheorem{conjecture}[theorem]{Conjecture}

\theoremstyle{definition}
\newtheorem{definition}[theorem]{Definition}

\theoremstyle{remark}
\newtheorem{remark}[theorem]{\bf Remark}

\numberwithin{equation}{section}

\newcommand{\C}{\mathbb C}

\newcommand{\Q}{\mathbb Q}
\newcommand{\G}{\mathbb G}
\newcommand{\Z}{\mathbb Z}

\begin{document}

\title[An algebraic Sato-Tate group and Sato-Tate conjecture]
{{An algebraic Sato-Tate group and Sato-Tate conjecture}}

\author[Grzegorz Banaszak]{Grzegorz Banaszak}
\address{Department of Mathematics and Computer Science, Adam Mickiewicz University,
Pozna\'{n} 61614, Poland}
\email{banaszak@amu.edu.pl}

\author[Kiran S. Kedlaya]{Kiran S. Kedlaya}
\address{Department of Mathematics, University of California, San Diego, La Jolla, CA 92093, USA}
\email{kedlaya@ucsd.edu}



\thanks{
Thanks to Jeff Achter for the reference to \cite{Ch}.
G. Banaszak was supported by the NCN (National Center for Science of Poland) 
grant NN 201 607 440 and UC San Diego during visit Dec. 2010 - June 2011. 
K.S. Kedlaya was supported 
by NSF CAREER grant DMS-0545904,
NSF grant DMS-1101343,
DARPA grant HR0011-09-1-0048, MIT (NEC Fund,
Cecil and Ida Green professorship), and UC San Diego
(Stefan E. Warschawski professorship). 
}
\begin{abstract}
We make explicit a construction of Serre giving a definition of an algebraic Sato-Tate group associated to an
abelian variety over a number field, which is conjecturally linked to the distribution of normalized $L$-factors 
as in the usual Sato-Tate conjecture for elliptic curves. 
The connected part of the algebraic Sato-Tate group is closely related to the Mumford-Tate group,
but the group of components carries additional arithmetic information. We then check that in many cases where the
Mumford-Tate group is completely determined by the endomorphisms of the abelian variety, the algebraic Sato-Tate
group can also be described explicitly in terms of endomorphisms. In particular, we cover all abelian varieties
(not necessarily absolutely simple)
of dimension at most 3; this result figures prominently in the analysis of
Sato-Tate groups for abelian surfaces given recently by Fit\'e, Kedlaya, Rotger, and Sutherland.
\end{abstract}

\maketitle


\section{Introduction}

Let $F$ be a number field with absolute Galois group $G_F$.
Let $A$ be an abelian variety of dimension $g \geq 1$ over $F$. 
Pick a prime number $l$ and consider the action of $G_F$ on the $l$-adic Tate module of $A$.
A theorem of Weil implies that for each prime ideal
$\mathfrak{p}$ of $F$ at which $A$ has good reduction, 
the characteristic polynomial $L(A/F, T)$ of any geometric Frobenius element of $G_F$ corresponding to $\mathfrak{p}$ 
is a monic polynomial
with integer coefficients, whose roots in the complex numbers all have absolute value $q^{1/2}$ for
$q$ the absolute norm of $\mathfrak{p}$.

One can ask how the renormalized characteristic polynomials $\overline{L}(A/F, T) = q^{-g} L(A/F, q^{1/2} T)$
are distributed in the limit as $q \to \infty$. For $A$ an elliptic curve without complex multiplication,
the \emph{Sato-Tate conjecture} predicts equidistribution with respect to the image of the Haar measure on
$\mathrm{SU}(2)$; this result is now known when $F$ is a totally real field \cite{BLGG11}, building upon
recent advances in the modularity of Galois representations led by R. Taylor. 
For $A$ an elliptic curve
with complex multiplication, the situation is easier to analyze: for all $F$, one can prove equidistribution
for the image of the Haar measure on a certain subgroup of $\mathrm{SU}(2)$. This group may be taken to be
$\mathrm{SO}(2)$ when $F$ contains the field of complex multiplication; otherwise, one must instead take the normalizer
of $\mathrm{SO}(2)$ in $\mathrm{SU}(2)$

With this behavior in mind, it is 
reasonable to formulate an analogous \emph{Sato-Tate conjecture} for general $A$, in which the equidistribution
is with respect to the image of the Haar measure on a suitable compact Lie group. Such a conjecture seems to have been
first formulated by Serre \cite{Se94} in the language of motives. More recently, Serre has given an alternate formulation in terms of $l$-adic Galois representations \cite[\S 8]{Se11}.

The purpose of this paper is to make explicit one aspect of Serre's $l$-adic construction of the putative \emph{Sato-Tate group},
namely the role of an algebraic group closely related to (but distinct from) the classical \emph{Mumford-Tate group}
associated to $A$. This \emph{algebraic Sato-Tate group} has the property that its connected part is conjecturally
determined by the Mumford-Tate group; however, whereas the Mumford-Tate group is by construction connected, the
algebraic Sato-Tate group has a component group with Galois-theoretic meaning.
For example, in those cases where the Mumford-Tate group is determined entirely by the endomorphisms of $A$
(e.g., the cases studied by the first author jointly with W. Gajda and P. Kraso{\' n} in \cite{BGK1, BGK2}), we can determine the algebraic Sato-Tate group
and interpret its component group as the Galois group of the minimal extension of $F$ over which all geometric
endomorphisms are defined.

One key motivation for this work is to establish the properties of the algebraic Sato-Tate group
for all abelian varieties of dimension at most 3. (Note that includes only cases where the Mumford-Tate group
is determined by endomorphisms, as the first counterexamples are Mumford's famous examples in dimension 4.)
This result is applied by the second author in \cite{FKRS} in order to classify the possible Sato-Tate groups
arising in dimension 2, and makes the corresponding classification in dimension 3 feasible as well.

In cases where the Mumford-Tate group is not determined by endomorphisms, one can still exhibit
a natural candidate for the algebraic Sato-Tate group by following Serre's motivic construction
\cite{Se94}. Based on computations in \S 7 and \S 8, using the Tannakian 
formalism, we define in \S 9 the algebraic Sato-Tate group $AST_K(A)$ assuming that 
the Mumford-Tate group is the connected component of the motivic Mumford-Tate group 
(a hypothesis put forward by Serre in \cite[sec.\ 3.4]{Se94}). Then in \S 9 we prove 
(see Theorem \ref{AST as expected extension of HA}) 
that the group $AST_K(A)$ satisfies the conditions (\ref{ASTExpectedProperties1})--(\ref{ASTExpectedProperties5})
of Remark (\ref{expected properties of the alg Sato-Tate group}) which one naturally expects to hold for the 
algebraic Sato-Tate group. If a suitable motivic version of the Mumford-Tate conjecture holds, 
then we prove that the algebraic Sato-Tate conjecture also holds. We also show that the $AST_K(A)$ group 
defined in \S 9 agrees with the algebraic Sato-Tate group considered in sections \S 4 - \S 6, in particular 
in all cases where the Mumford-Tate group is determined by endomorphisms.


\section{The algebraic Sato-Tate group and conjecture}

We begin with the definition of the algebraic Sato-Tate group and the formulation of an associated conjecture refining
the standard Mumford-Tate conjecture.

\begin{definition}
Let $A/F$ be an abelian variety over a number field $F.$ Choose an 
embedding $F \subset \C$, then put $V := V(A) := H_1(A(\C), \Q)$ and 
$V^{\ast} := {\rm Hom}_{\Q}(V, \, \Q).$ 
Put $D := D(A) := {{\rm End}} (A)^{0}:= {{\rm End}}_{{\overline F}} (A) \otimes \Q$.

We obtain a rational Hodge structure of weight 1 on $V^{\ast}$ from the
Hodge decomposition
$$V^{\ast}_{\C} := V^{\ast} \otimes_{\Q} {\C}=H^{1,0}\oplus H^{0,1},$$
where \,\,$H^{p,q} = H^p(A;\, \Omega_{A/\C}^q)$ and ${\overline{H^{p,q}}}= H^{q,p}.$ Observe that
$H^{p,q}$ are invariant subspaces with respect to the action of $D$ on $V^{\ast} \otimes_{\Q} \C$, so the $H^{p,q}$ are right $D$-modules. 
Put $H^{-p, -q} := {\rm{Hom}}_{\C}(H^{p,q}, \, \C).$ The left action of $D$ on
$H^{-p, -q}$ is defined  via the right action of $D$ on $H^{p,q}.$
Let
$$\psi = \psi_{V} :\, V \times V\, \rightarrow \Q$$
be the $\Q$-bilinear,  nondegenerate, alternating form coming from the Riemann form of $A.$ 
\end{definition}

\begin{definition}
Put $T_l := T_l (A)$ and $V_l := V_l (A) := T_l (A) \otimes_{\Z_l} \Q_l.$ 
It is well known 
\cite[\S 24, p. 237]{M} (cf.\
\cite[\S 16, p. 133]{Mi}) that $V_l \cong V \otimes_{\Q} \Q_l.$ 
The Galois group $G_F$ acts naturally on $V_l$ and
because of the Weil pairing $\psi_{l} \, :\, V_l \times V_l \rightarrow \Q_l$
we have the $l$-adic representation
\begin{equation}
\rho_l \, :\, G_F \rightarrow GSp (V_{l})
\end{equation}
induced by the natural representation $\rho_l \, :\, G_F \rightarrow GSp (T_{l}).$

Let $K/F$ be a finite extension.
Let $G_{l, K}^{alg} \subseteq  GSp_{(V_{l}, \psi_{l})}$ 
be the Zariski closure of $\rho_{l} (G_K)$ in $GSp_{(V_{l}, \psi_{l})}.$
Define 
$$\rho_{l} (G_K)_{1} := \rho_{l} (G_K) \cap Sp_{(V_{l}, \psi_{l})}.$$
Let $G_{l, K, 1}^{alg} \subseteq  Sp_{(V_{l}, \psi_{l})}$
be the Zariski closure of $\rho_{l} (G_K)_{1}$ in 
$Sp_{(V_{l}, \psi_{l})}.$
The algebraic group $G_{l, K}^{alg}$ is reductive by \cite[Theorem 3]{F}.
If $K = F$, we will put $G_{l}^{alg} := G_{l, F}^{alg}$ and
$G_{l, 1}^{alg} := G_{l, F, 1}^{alg}.$ 
\end{definition}

\noindent
On the basis of the work of Mumford, Serre,
and Tate, it has been conjectured for about 50 years (see Conjecture~\ref{Mumford-Tate})
that there is a connected reductive group scheme 
$G \subseteq GSp_{(V,\psi)}$ such that for all $l \gg 0$,
\begin{equation}
G_{\Q_l} = (G_{l, K}^{alg})^{\circ},
\label{previous to Algebraic Sato-Tate equality}
\end{equation}
\noindent
where $(G_{l,K}^{alg})^{\circ}$ denotes the connected component of the
identity.
This implies that $G$ does not depend on $K$ (see \S\ref{section:Zariski closure}).
We propose the following refinement of this conjecture which does detect the field $K$.
\begin{conjecture}\label{algebraic Sato Tate conj.} (Algebraic Sato-Tate conjecture)
There is an algebraic group $AST_K (A) \subseteq Sp_{(V, \psi)}$ over $\Q$ such that  $AST_K (A)^{\circ}$ is reductive and 
for each $l$,
\begin{equation}
AST_K(A)_{\Q_l} = G_{l, K, 1}^{alg}.
\label{Algebraic Sato-Tate equality}
\end{equation}
\end{conjecture}

\begin{definition}
The algebraic Sato-Tate group plays a key role in Serre's definition of the putative Sato-Tate group
$ST_{K} (A).$ 
If Conjecture \ref{algebraic Sato Tate conj.} holds, then
base extension of $AST_K(A)_{\Q_l}$ along an embedding $\Q_l \subset \C$ 
gives a group $AST_K(A)_{\C}$ that does not depend on $l$ or the embedding. Taking a maximal
compact subgroup of $AST_K(A)(\C)$, we obtain the \emph{Sato-Tate group} $ST_K(A) 
\subseteq USp(V_{\C}).$
\end{definition}

\begin{remark}
Besides Conjecture 
\ref{algebraic Sato Tate conj.} and the usual Sato-Tate conjecture,
we expect $AST_K (A)$ to have the following properties:
\begin{gather}
AST_{K} (A) \subseteq  DL_K(A)
\label{ASTExpectedProperties1} \\
AST_{K}(A)^{\circ} = H(A)
\label{ASTExpectedProperties3} \\
\pi_{0} (AST_{K}(A)) = \pi_{0} (ST_{K}(A)).
\label{ASTExpectedProperties5}\end{gather}  
Here $H(A)$ is the Hodge group and $DL_K(A)$ is the twisted decomposable Lefschetz 
group, to be defined in \S 5. 
In \S 5, we exhibit a wide class of abelian varieties 
such that Conjecture \ref{algebraic Sato Tate conj.} and conditions 
(\ref{ASTExpectedProperties1})--(\ref{ASTExpectedProperties5}) hold 
(see Theorem \ref{conditions for AST}). 

In \S 9, we define 
a candidate group $AST_K(A)$ based on computations in \S 7 and \S 8 using the Tannakian 
formalism. Assuming that the Mumford-Tate group is the connected 
component of the motivic Mumford-Tate group (a hypothesis 
put forward by Serre in \cite[sec.\ 3.4]{Se94}), we prove
(see Theorem \ref{AST as expected extension of HA})
that the group $AST_K(A)$ satisfies conditions 
(\ref{ASTExpectedProperties1})--(\ref{ASTExpectedProperties5}).
\label{expected properties of the alg Sato-Tate group}\end{remark}


\section{Some remarks on Zariski closure. A theorem of Serre.}
\label{section:Zariski closure}

Before proceeding further, we collect some observations about the groups we have just defined.

\begin{remark}
We have the following commutative diagram with exact columns: 
$$
\xymatrix{
\rho_{l} (G_K)_{1} \ar@<0.1ex>[d]^{} \ar[r]^{}  \quad 
& \quad G_{l, K, 1}^{alg} \ar@<0.1ex>[d]^{} \ar[r]^{}   \quad & 
\quad Sp_{(V_{l}, \psi_{l})} \ar@<0.1ex>[d]^{} \\ 
\rho_{l} (G_K) \ar@<0.1ex>[d]^{} \ar[r]^{}  \quad 
& \quad G_{l, K}^{alg} \ar@<0.1ex>[d]^{} \ar[r]^{}   \quad & 
\quad  GSp_{(V_{l}, \psi_{l})} \ar@<0.1ex>[d]^{} \\
\quad \Q_{l}^{\times}  \ar[r]^{}  \quad & \quad \G_{m}  
\ar[r]^{=}  \quad & 
\quad  \G_{m} \\
}
\label{initial diagram}$$
By the theorem of Bogomolov \cite{Bo} on homotheties, 
$\rho_l(G_K)$ is an open subgroup of $G_{l,K}^{alg}$.
Since $Sp_{(V_{l}, \psi_{l})}$ is closed in $GSp_{(V_{l}, \psi_{l})}$, 
$\rho_l(G_K)_1 = \rho_l(G_K) \cap Sp_{(V_{l}, \psi_{l})}$ is an open subgroup of $G_{l, K}^{alg} \cap Sp_{(V_l, \psi_l)}$.
It follows that 
\begin{equation}
G_{l, K, 1}^{alg} = G_{l, K}^{alg} \cap Sp_{(V_l, \psi_l)}
\label{A property of GlK1alg }
\end{equation}
is closed in $G_{l, K}^{alg}$, and that there is an exact sequence
\begin{equation}
1 \rightarrow G_{l, K, 1}^{alg} \rightarrow  G_{l, K}^{alg} \rightarrow \G_m \rightarrow 1.
\label{The exact sequence for GlK1alg and GlKalg}
\end{equation}
Since $G_{l, K, 1}^{alg}$ is the kernel of a homomorphism from a reductive group to a torus, it is also reductive.
\end{remark}

\begin{remark}
Let $F \subseteq K \subseteq L \subset \overline{F}$ be a chain of extensions such that
$K/F$  and $L/K$ are finite. Let $L' \subset \overline{F}$ be the Galois closure of $L$ over $K$. It is clear that 
$G_{l, L'}^{alg} \subseteq G_{l, L}^{alg} \subseteq G_{l, K}^{alg}$ and $G_{l, L', 1}^{alg}  \subseteq G_{l, L, 1}^{alg}  \subseteq G_{l, K, 1}^{alg}.$ 
On the other hand, there is a surjective homomorphism
$\rho_l(G_K)/\rho_l(G_{L'}) \to G_{l,K}^{alg}/G_{l,L'}^{alg}$, so $G_{l,L'}^{alg}$ is a subgroup of $G_{l,K}^{alg}$
of finite index, as then is $G_{l,L}^{alg}$. 
In particular, $(G_{l,L}^{alg})^\circ = (G_{l,K}^{alg})^\circ$. By a similar argument,
$(G_{l,L,1}^{alg})^\circ = (G_{l,K,1}^{alg})^\circ$.
\end{remark}

\noindent
\begin{proposition}
There is a finite Galois extension $L_{0}/K$ such that 
$G_{l, L_{0}}^{alg} = (G_{l, K}^{alg})^{\circ}$ and $G_{l, L_{0}, 1}^{alg} = (G_{l, K, 1}^{alg})^{\circ}.$ 
\label{L0realizing conn comp}\end{proposition}

\begin{proof}
Let $Z$ be the set of $\Q_l$-points of the closed subscheme
$(G_{l, K}^{alg} - (G_{l, K}^{alg})^{\circ}) \cup (G_{l, K,1}^{alg} - (G_{l, K, 1}^{alg})^{\circ})$
of $GSp_{(V_l,\psi_l)}$. Since $\rho_l$ is continuous,
we can find a finite Galois extension $L_0/K$ such that
$\rho_l(G_{L_0}) \cap Z = \emptyset$.
Such an extension satisfies
$G_{l, L_{0}}^{alg} \subseteq (G_{l, K}^{alg})^{\circ}$ and $G_{l, L_{0}, 1}^{alg} \subseteq (G_{l, K, 1}^{alg})^{\circ}.$ 
Since we already have the reverse inclusions, we obtain the desired equalities.
\end{proof}

\noindent
The following theorem is a special case of a result of Serre \cite[\S 8.3.4]{Se11}.
\begin{theorem}\label{equality of conn comp for Glalg and Glalg1} 
The following map is an isomorphism:
$$i_{CC} \,\, : \,\,  G_{l, K, 1}^{alg} /(G_{l, K, 1}^{alg})^{\circ} \,\, 
{\stackrel{\cong}{\longrightarrow}} \,\, G_{l, K}^{alg} /(G_{l, K}^{alg})^{\circ}.$$  
\end{theorem}
\begin{proof} 
Choose $L_0$ as in Proposition~\ref{L0realizing conn comp}.
Then the following commutative diagram has exact rows:
$$
\xymatrix{
\quad & \quad 1 \ar@<0.1ex>[d]^{}  \quad & \quad 1 \ar@<0.1ex>[d]^{} \quad & 
1 \ar@<0.1ex>[d]^{} \\ 
1 \ar[r]^{} \quad & \quad G_{l, L_0, 1}^{alg} \ar@<0.1ex>[d]^{} \ar[r]^{}  \quad 
& \quad  G_{l, K, 1}^{alg} \ar@<0.1ex>[d]^{} \ar[r]^{}  \quad 
& \quad G_{l, K, 1}^{alg} /(G_{l, K, 1}^{alg})^{\circ}   
\ar@<0.1ex>[d]^{i_{CC}} \ar[r]^{}  \quad 
& \,\, 1 \\ 
1 \ar[r]^{} \quad & \quad G_{l, L_0}^{alg} \ar@<0.1ex>[d]^{} \ar[r]^{}  \quad 
& \quad  G_{l, K}^{alg} \ar@<0.1ex>[d]^{} \ar[r]^{}  \quad 
& \quad G_{l, K}^{alg} /(G_{l, K}^{alg})^{\circ}   \ar@<0.1ex>[d]^{} \ar[r]^{}  \quad 
&  \,\, 1 \\
1 \ar[r]^{} \quad & \quad \G_m  \ar@<0.1ex>[d]^{}  \ar[r]^{=}  
\quad & \quad \G_m  \ar@<0.1ex>[d]^{}  \ar[r]^{} \quad & 1\\
 \quad & \quad 1   \quad & 1 \\}
\label{diagram in Serre theorem}$$
\noindent
The first two columns are also exact by \eqref{The exact sequence for GlK1alg and GlKalg},
so a diagram chase (as in the snake lemma) shows that the third column is also exact.
\end{proof}

\begin{remark}
Observe that the natural action by left translation
\begin{equation}
G_K \times G_{l, K}^{alg} /(G_{l, K}^{alg})^{\circ} \rightarrow G_{l, K}^{alg} /(G_{l, K}^{alg})^{\circ} 
\label{ContLeftTransOnConnecComp}\end{equation}
is continuous. By Theorem \ref{equality of conn comp for Glalg and Glalg1}, we also have a continuous action
\begin{equation}
G_K \times G_{l, K,1}^{alg} /(G_{l, K,1}^{alg})^{\circ} \rightarrow G_{l, K,1}^{alg} /(G_{l, K,1}^{alg})^{\circ}.
\end{equation}
\end{remark}

\noindent
Choose a suitable field embedding $\Q_l \rightarrow \C$ and
put  $ {G_{l, K, 1}^{alg}}_{\C} := G_{l, K, 1}^{alg} \otimes_{\Q_l} \C.$ 
By Theorem \ref{equality of conn comp for Glalg and Glalg1} and  
\cite[Lemma~2.8]{FKRS}, we have the following.

\begin{proposition}\label{connected components iso} 
Let $K/F$ be any finite extension such that $F \subseteq K \subset \overline{F}.$
Then there are natural isomorphisms
\begin{equation}
G_{l, K, 1}^{alg} /(G_{l, K, 1}^{alg})^{\circ} \quad
\cong \quad {G_{l, K, 1}^{alg}}_{\C} \,\, /({G_{l, K, 1}^{alg}}_{\C})^{\circ}
\quad \cong \quad 
ST_{A} / ST_{A}^{\circ}.
\label{ConnCompIsom}\end{equation}
If the algebraic Sato-Tate conjecture (Conjecture~\ref{algebraic Sato Tate conj.})
holds, then all of the groups in (\ref{ConnCompIsom}) are isomorphic to
$AST_{K} (A) /AST_{K} (A)^{\circ}.$
\end{proposition}

\begin{remark}\label{minimal field of connectedness}
Consider the continuous homomorphism
\begin{equation}
\epsilon_{l, K}\, :\, G_K \rightarrow G_{l, K}^{alg}(\Q_l). 
\nonumber\end{equation}
Serre \cite{Se84}, \cite{Se81} proved that $\epsilon_{l, K}^{-1} ((G_{l, K}^{alg})^{\circ}(\Q_l))$ is independent of $l$.
Since $(G_{l, K}^{alg})^{\circ}$ is open in $G_{l, K}^{alg}$, this preimage 
equals $G_{L_0} = G(\overline{K}/L_{0})$ for some finite Galois extension $L_{0}/K.$
This theorem of Serre can also be seen from Propositions
\ref{L0realizing conn comp} and \ref{connected components iso} and Theorem 
\ref{equality of conn comp for Glalg and Glalg1}.
In fact, $L_{0}/K$ is the minimal extension such that 
$G_{l, L_{0}}^{alg}$ and $G_{l, L_{0}, 1}^{alg}$ are connected for
all $l.$ 
Let $H_{l, K, 1} := \rho_{l}^{-1} (\rho_{l} (G_K)_1).$ Put $K_1 := \overline{K}^{H_{l, K, 1}}.$
Observe that
\begin{gather*}
\epsilon_{l, K}^{-1} ((G_{l, K, 1}^{alg})^{\circ}(\Q_l)) = \epsilon_{l, K}^{-1} (G_{l, L_0, 1}^{alg}(\Q_l)) =
\epsilon_{l, K}^{-1} ((G_{l, L_0}^{alg} \cap Sp_{(V_l, \psi_l)})(\Q_l)) = 
\\
= \epsilon_{l, K}^{-1} (G_{l, L_0}^{alg}(\Q_l)) \,\cap\, 
\epsilon_{l, K}^{-1}(Sp_{(V_l, \psi_l)}(\Q_l))
= G_{L_0} \cap G_{K_1} = G_{LK_1}.
\end{gather*}
\end{remark}


\section{Mumford-Tate group and Mumford-Tate conjecture}

We next recall some standard definitions and results concerning the Mumford-Tate group.
We continue to assume that $A$ is an abelian variety over a number field $F$ equipped with a fixed
embedding into $\C$.

\begin{definition}
For $V = H_1 (A,\, \Q)$,  define the cocharacter
$$\mu_{\infty, A}:\mathbb{G}_{m}({\C})\rightarrow GL(V_{\C})$$ 
such that for any $z\in {\C}^{\times},$ the automorphism $\mu_{\infty, A}(z)$ acts as multiplication by
$z$ on $H^{-1,0}$ and as the identity on $H^{0,-1}.$ 
Observe that by definition we have
$\mu_{\infty, A}(\C^{\times}) \, \subset GSp_{(V,\, \psi)} (\C).$
\end{definition}

\begin{definition}\label{Definition Mumford Tate} (Mumford-Tate and Hodge groups)
\newline

\begin{enumerate}
\item{}
The \emph{Mumford-Tate group} of $A/F$ is the smallest algebraic subgroup 
$MT(A) \subseteq GL_{V}$ over $\Q$ such that $MT(A)({\C})$
contains $\mu_{\infty}({\C}).$ 
\item{} The \emph{decomposable Hodge group} is $D H(A) := MT(A) \cap SL_{V}$.
\item{}
The \emph{Hodge group} $H(A) := D H(A)^{\circ}$ is the connected
component of the identity in $D H(A).$
\end{enumerate}
By definition, if $K/F$ is a field extension
such that $K \subset \overline{F}$, then $MT (A)$ and $H (A)$ for $A/F$ 
and for $A/K$ are the same. That is, we do not need to include $K$ in the notation.
\end{definition}

\begin{remark}
Note that 
$MT(A)$ is a reductive subgroup of $GSp_{(V,\, \psi)}$ \cite[Prop. 3.6]{D},
$H(A)\, \subseteq Sp_{(V,\, \psi)}$,
and $MT(A) \subseteq C_{D}(GSp_{(V,\, \psi)}).$ Hence 
\begin{equation}
H(A) \subseteq C_{D}(Sp_{(V,\, \psi)}).
\label{DHA and CD}\end{equation}
\end{remark}
  
\begin{definition}\label{Definition Lefschetz} 
The algebraic group $L(A) = C_D^{\circ}(Sp_{(V,\, \psi)})$ is called
the {\it Lefschetz group}\footnote{This terminology is due to Murty \cite{Mu}.} of $A.$ By \eqref{DHA and CD} and the connectedness of $H(A)$, we have
$H(A) \subseteq L(A)$.
\end{definition}

\begin{definition}
For a field extension $L/\Q$ put
$$MT(A)_{L} := MT(A) \otimes_{\Q} L,\quad D H(A)_{L} := D H(A) \otimes_{\Q} L, $$
$$ H(A)_{L} := H(A) \otimes_{\Q} L,\quad L(A)_{L} := L(A) \otimes_{\Q} L.$$
\end{definition}

\begin{theorem}
\label{Theorem Deligne} 
(Deligne \cite[I, Prop.\ 6.2]{D})
For any prime number $l$,
\begin{equation}
(G_l^{alg})^{\circ} \subseteq MT(A)_{\Q_l}.
\label{Del ineq}\end{equation}
\end{theorem}

\begin{conjecture} (Mumford-Tate) \label{Mumford-Tate} 
For any prime number $l$,
\begin{equation}
(G_l^{alg})^{\circ} = MT(A)_{\Q_l}.
\label{MT eq}\end{equation}
\end{conjecture}

\begin{remark}\label{Mumford-Tate Conj and Deligne Theorem} 
Observe that (\ref{Del ineq}) is equivalent to the inclusion
\begin{equation}
(G_{l, K, 1}^{alg})^{\circ} \subseteq H (A)_{\Q_l},
\label{Del ineq1}\end{equation}
while the Mumford-Tate conjecture is equivalent to the equality
\begin{equation}
(G_{l, K, 1}^{alg})^{\circ} = H(A)_{\Q_l}.
\label{H eq}\end{equation}
This follows immediately from the following commutative diagram
in which every column is exact and every horizontal arrow is a containment 
of corresponding group schemes.
$$
\xymatrix{
\quad 1 \ar@<0.1ex>[d]^{}  \quad 
& \quad 1 \ar@<0.1ex>[d]^{}  \quad 
&  1\ar@<0.1ex>[d]^{}\\
G_{l, K, 1}^{alg} \ar@<0.1ex>[d]^{} \ar[r]^{}  \quad 
& \quad D H(A)_{\Q_l}   
\ar@<0.1ex>[d]^{} \ar[r]^{}  \quad & 
\quad Sp_{V_{l}, \psi_{l}} \ar@<0.1ex>[d]^{} \\ 
G_{l, K}^{alg} \ar@<0.1ex>[d]^{} \ar[r]^{} \quad & 
\quad MT (A)_{\Q_l}  \ar@<0.1ex>[d]^{} \ar[r]^{}  \quad & 
\quad  GSp_{V_{l}, \psi_{l}} \ar@<0.1ex>[d]^{} \\
\quad \G_{m} \ar[r]^{=} \ar@<0.1ex>[d]^{}  \quad & 
\quad  \G_{m} \ar[r]^{=} \ar@<0.1ex>[d]^{}  \quad & 
\quad  \G_{m} \ar@<0.1ex>[d]^{} \\
\quad 1   \quad & \quad 1   \quad &  1}
\label{diagram in Deligne theorem}
$$
\end{remark}

It is known that the Mumford-Tate and Hodge groups do not behave well
in general with respect to products of abelian varieties
\cite[p. 316]{G}. However one can prove the following theorem.

\begin{theorem}\label{Mumford-Tate group standard properties} 
The Mumford-Tate group of abelian varieties have the following properties.
\begin{itemize} 
\item[1.] An isogeny $\phi : A_1 \rightarrow A_2$ of abelian 
varieties induces isomorphisms
$MT (A_1) \cong MT (A_2)$  and $H (A_1) \cong H (A_2).$ 
\item[2.] If $A$ is an abelian variety and $A^s := \prod_{i=1}^s A$,
then $MT (A^s) \cong MT (A)$ and $H (A^s) \cong H (A)$.
\end{itemize}
\end{theorem}
\begin{proof}
1. This follows because an isogeny induces an isomorphism of associated
polarised rational Hodge structures.

\noindent
2. Observe that there is a natural isomorphism $H^1 (A^s (\C), \, \Q) \cong 
\oplus_{i=1}^s H^1 (A (\C), \, \Q)$ or simply $V (A^s) \cong V(A)^s.$  
So in this case, the Hodge 
structure on $H^1 (A^s (\C), \, \Q)$ is isomorphic to the direct sum of the
Hodge structures on $H^1 (A (\C), \, \Q).$ Via this identification, we have
$\mu_{\infty, A^s} = \Delta (\mu_{\infty, A})$ where 
$\Delta\, :\, GSp (V_{\C}) \rightarrow GSp (V_{\C})^s \subseteq GL (V_{\C}^{s})$
is the diagonal homomorphism. So we can consider 
$MT (A^s)$ as a subgroup of $MT (A)^s := \prod_{i=1}^s MT (A).$
Consider the diagonal homomorphism
$\Delta \, :\, MT (A) \rightarrow MT (A)^s.$  
Since $\mu_{\infty, A^s}(\G_m (\C)) = 
\Delta (\mu_{\infty, A}) (\G_m (\C)) \subseteq \Delta (MT (A)) (\C)$,
we have
$MT (A^s) \subseteq  \Delta (MT (A)) \cong MT (A).$ Since 
$MT (A^s)(\C)$ contains $\Delta (\mu_{\infty, A}(\G_m (\C)))$ then the image of 
$MT (A^s)$ via projection $\pi \, :\, MT (A)^s \rightarrow MT (A)$ on every factor  
contains the image of $\mu_{\infty, A}(\G_m (\C)).$ By the definition of $MT(A)$, this means that 
the projection $\pi$ must surject onto $MT (A).$ On the other hand, since $MT (A^s) \subseteq  \Delta (MT (A))$,
it is clear that for $\pi$ to be onto, we must have $MT (A^s) = \Delta (MT (A))$.
To get the isomorphism of Hodge groups, we note that the isomorphism $V (A^s) \cong V(A)^s$  
gives an isomorphism $SL_{V(A^s)} \cong SL_{V(A)^s}.$ We thus have  the following isomorphisms 
of algebraic groups over $\Q$:
\begin{gather*}
MT (A^s) \cap  SL_{V(A^s)} \cong \Delta (MT (A)) \cap SL_{V(A)^s} \cong \\
\cong \Delta (MT (A)) \cap (SL_{V(A)})^s \cong  MT (A) \cap SL_{V(A)}.
\end{gather*}
Taking connected components, we get the desired isomorphism of Hodge groups. 
\end{proof}

One can make a corresponding calculation also on the Galois side.
\begin{theorem}\label{Zariski closure standard properties} 
We have the following results.
\begin{itemize} 
\item[1.] An isogeny $\phi : A_1 \rightarrow A_2$ of abelian varieties over a number field $K$ induces isomorphisms 
$G_{l, K}^{alg} (A_1) \cong 
G_{l, K}^{alg} (A_2)$ and 
$G_{l, K, 1}^{alg} (A_1) \cong 
G_{l, K, 1}^{alg} (A_2).$ 
\item[2.] If $A$ is an abelian variety over a number field $K$, then
for any positive integer $s$,
$G_{l, K}^{alg} (A^s) \cong G_{l, K}^{alg}  (A)$ and 
$G_{l, K, 1}^{alg}  (A^s) = G_{l, K, 1}^{alg}  (A).$
\end{itemize}
\end{theorem}

\begin{proof}
1. Observe that an isogeny $\phi$ induces an isomorphism of Galois modules
$V_l (A_1) \cong V_l (A_2).$

\noindent
2. The natural isomorphism $V_l (A^s) \cong V_l (A)^s$ of $G_K$-modules shows that
$\rho_{l, A^s} \cong \Delta \rho_{l, A}$, where $\Delta \rho_{l, A} \, :\, G_K \rightarrow 
GSp (V_l (A))^s$ is the natural diagonal representation
$\Delta \rho_{l, A} = {{\rm diag}} (\rho_{l, A}, \dots, \rho_{l, A}).$
\end{proof}

\begin{corollary}\label{MT true for A implies true for As}
Let $A/K$ be an abelian variety such that the Mumford-Tate conjecture
holds for $A.$ Then the Mumford-Tate conjecture
holds for $A^s$ for any positive integer $s.$
\end{corollary}
\begin{proof} It follows from Theorems \ref{Mumford-Tate group standard properties}  
and \ref{Zariski closure standard properties}. 
\end{proof}

\begin{remark}\label{Mumford-Tate Conj AST conjecture} 
Observe that if the Mumford-Tate conjecture holds for $A$
and $K$ is a finite extension of $F$ for which $G_{l, K}^{alg}$ is connected, then
\begin{equation}
G_{l, K, 1}^{alg}(A^s) = H(A^s)_{\Q_l}.
\label{H eq1}\end{equation}
for any $s \geq 1.$
Hence the algebraic Sato-Tate conjecture holds for $A^s$ for any $s \geq 1$ with
\begin{equation}
AST_{K} (A^s) = H (A^s).
\label{AST conj holds when MT conj and GL ALG connect }\end{equation}
\end{remark}


\section{Twisted Lefschetz groups and the algebraic Sato-Tate conjecture}

We next use the Lefschetz group to define an upper bound on the algebraic Sato-Tate group.
In those cases where the Mumford-Tate group is determined entirely by endomorphisms, this
allows us to deduce the algebraic Sato-Tate conjecture from the Mumford-Tate conjecture.
(An analogous construction for modular forms is the \emph{twisting field} appearing
in the work of Momose \cite{Mo} and Ribet \cite{R}.)
We continue to take $A$ to be an abelian variety over a number field $F$ embedded in $\C$
and $K$ to be a finite extension of $F$.
 

\begin{definition}
Note that we have a continuous representation
\begin{equation}
\rho_{e} \, : \, G_K \rightarrow Aut_{\Q} (D).
\end{equation}
Put $G_{L_{e}} := \,\, \text{Ker}\, \rho_{e},$ so that $L_{e}/K$ is a finite Galois extension.
For $\tau \in G(L_{e}/K)$, define
\begin{equation}
DL_{K}^{\tau}(A) := \{g \in Sp_{V}: \, g \beta g^{-1} = \rho_e(\tau) (\beta) \,\,\,
\forall \,\beta \in D\}.
\label{decomposable twisted Lefschetz for fixed element}\end{equation} 
It is enough to impose the condition for $\beta$ running over a $\Q$-basis of $D$;
therefore, $DL_{K}^{\tau}(A)$ is a closed subscheme of $Sp_V$ for each $\tau.$ 
\end{definition}

\begin{definition}
Define the \emph{twisted decomposable algebraic Lefschetz group} of $A$ over $K$
to be the closed algebraic subgroup of $Sp_{V,\psi}$ given by
\begin{equation}
DL_{K}(A) = \bigsqcup_{\tau \in G(L_{e}/K)} \,\, DL_{K}^{\tau}(A).
\label{decomposition into twisted Lefschetz for fixed elements}\end{equation} 
\noindent
For any subextension $L/K$ of ${\overline F}/K$ we have
$DL_L(A) \subseteq DL_K(A)$ and $DL_{L}^{id}(A) = DL_{K}^{id}(A)$.
Hence the classical Lefschetz group of $A$ can be written as:
\begin{equation}
L(A) = DL^{id}_{K}(A)^{\circ} = DL_{K}(A)^{\circ}.
\label{twisted Lefschetz}\end{equation}
and $L(A) = DL^{id}_{L}(A)^{\circ} = DL_{L}(A)^{\circ},$
for any subextension $L/K$ of ${\overline F}/K.$
In particular $DL_{L_e}^{id}(A) = DL_{L_e}(A) = 
DL_{\overline F}(A) = DL_{{\overline F}}^{id}(A)$ and
$L (A) = DL_{L_e}(A)^{\circ}= DL_{\overline F}(A)^{\circ}.$
\end{definition}

\begin{theorem}\label{Decomposible Lefschetz group standard properties} 
The twisted decomposable Lefschetz groups of abelian varieties have the following 
properties.
\begin{itemize} 
\item[1.] An isogeny $\phi : A_1 \rightarrow A_2$ of abelian 
varieties over $K$ induces an isomorphism
$DL_{K}^{\tau} (A_1) \cong DL_{K}^{\tau} (A_2)$ for every $\tau \in G(L_e /K)$, and  
consequently an isomorphism $DL_{K} (A_1) \cong DL_{K} (A_2).$
\item[2.] If $A / K$ is an abelian variety, then $DL_{K}^{\tau} (A^s) \cong DL_{K}^{\tau} (A)$
for every $\tau \in G(L_e /K).$
\item[3.] If $A = \prod_{i=1}^t A_{i}$ with $A_i / K$ simple for all $i$ and 
${\rm{Hom}}_{\overline{F}} (A_i, \, A_j) = 0$
for all $i \not= j,$ then $DL_{K}^{\tau} (A) \cong \prod_{i=1}^t \, DL_{K}^{\tau} (A_i)$
for every $\tau \in G(L_e /K).$ 
\item[4.] If $A = \prod_{i=1}^t A_{i}^{s_i}$ with $A_i / K$ for all $i$ and 
${\rm{Hom}}_{\overline{F}} (A_i, \, A_j) = 0$
for all $i \not= j,$ then $DL_{K}^{\tau} (A) \cong \prod_{i=1}^t \, DL_{K}^{\tau} (A_i)$
for every $\tau \in G(L_e /K).$  
\item[5.] If $A = \prod_{i=1}^t A_{i}^{s_i}$ with all $A_i / K$ simple, pairwise 
nonisogenous then $DL_{K}^{\tau} (A) \cong \prod_{i=1}^t \, DL_{K}^{\tau} (A_i)$
for every $\tau \in G(L_e /K).$ 
\end{itemize}
\end{theorem}
\begin{proof}
1. An isogeny induces the isomorphism 
$(V (A_1), \psi_{V (A_1)}) \cong (V (A_2), \psi_{V (A_2)})$
of $\Q$-bilinear spaces.
It also induces an isomorphism 
$D (A_1) = {\rm{End}}_{\overline{K}} (A_1) \otimes_{\Z} \Q \cong 
D (A_2) = {\rm{End}}_{\overline{K}} (A_2) \otimes \Q$
of $\Q[G_K]$-modules, and so the first claim follows. 

\noindent
2. There is a natural isomorphism:
$$(V (A^s), \psi_{V (A^s)}) \cong (V (A), \psi_{V (A)})^s := \bigoplus_{i=1}^s (V (A), \psi_{V (A)})$$
and an isomorphism  $D(A^s) \cong 
M_{s} (D (A))$ of $\Q$-algebras. Let $\Delta$ be the homomorphism that maps 
$Sp_{(V (A), \psi_{V (A)})}$ onto 
${{\rm diag}} (Sp_{(V (A), \psi_{V (A)})}, \dots,
Sp_{(V (A), \psi_{V (A)})}) \subseteq Sp_{(\oplus_{i=1}^s (V (A), \psi_{V (A)}))}.$
Since $\Q \subseteq D(A)$, we have $M_{s} (Q) \subseteq M_{s} (D(A))$; from the 
definition of the twisted decomposable Lefschetz group, we get
$DL_{K}^{\tau} (A^s) \cong \Delta (DL_{K}^{\tau} (A)) \cong DL_{K}^{\tau} (A)$.

\noindent
3. The proof is similar to the proof of 2.\ under the observation that
$D(A) \cong \prod_{i=1}^t \, D(A_i)$ and 
$(V (A), \psi_{V (A)}) \, \cong \, \bigoplus_{i=1}^t \, (V (A_i), \psi_{V (A_i)}).$

\noindent
4. This follows immediately from 2.\ and 3. 

\noindent
5. Immediate corollary from 4.
\end{proof}

\begin{remark}\label{Lefschetz group standard properties} 
Theorem \ref{Decomposible Lefschetz group standard properties} 
remains true if we replace $DL_{K}^{\tau} (B)$ with $DL_{K}^{\tau} (B)^{\circ}$
for all abelian varieties $B$ that appear in the theorem. Since 
$L (B) = DL_{K}^{id} (B)^{\circ}$, then the Lefschetz group
satisfies properties 1-5 of Theorem 
\ref{Decomposible Lefschetz group standard properties}.
The property 5.\ of this theorem in the  case of the Lefschetz group was obtained 
previously by K. Murty \cite[Lemma 2.1]{Mu}, cf.\ \cite[Lemma B.70]{G}.
\end{remark}

\begin{remark}\label{decomposible bigger twisted Lefschetz}
Observe that we have
\begin{equation}
DL_K(A) := \{g \in Sp_{(V, \psi)}: \, \exists_{\, \tau \in G_K} \, 
\forall_{\, \beta \, \in D} \quad  
g \beta g^{-1} = \rho_{e} (\tau) (\beta) \, \}
\nonumber\end{equation} 
Changing quantifiers we get another group scheme
\begin{equation}
\widetilde{DL_K(A)} := 
\{g \in Sp_{(V, \psi)}: \, \forall_{\, \beta \, \in D} \, \exists_{\, \tau \in G_K}  
\quad g \beta g^{-1} = \rho_{e} (\tau) (\beta) \, \}
\label{decomposable bigger twisted Lefschetz1}\end{equation} 
Observe that 
$DL_K(A) \subseteq \widetilde{DL_K(A)}.$
\end{remark}

We now use the twisted Lefschetz group to obtain an upper bound on the algebraic Sato-Tate group.
\begin{remark}
Observe that (\ref{DHA and CD}) implies that
\begin{equation}
D H(A) \subseteq DL^{id}_{K}(A) \subseteq DL_K(A).
\label{DHA and DLKA}
\end{equation}
\noindent
On the other hand,
for any $P \in A({\overline F})$, any $\beta \in End_{{\overline F}} (A)$, and any 
$\sigma \in G_K$, we have $\sigma \beta \sigma^{-1} (P) = \sigma (\beta) (P).$
Hence
$\rho_{l} (G_K)_1 \subseteq DL_K(A) (\Q_l).$ In particular,
\begin{equation}
G_{l, K, 1}^{alg} \subseteq DL_K(A)_{\Q_l}.
\end{equation}
\end{remark}

We will be particularly interested in cases where 
$G_{l, K, 1}^{alg} = DL_K(A)_{\Q_l}$. We next check that this condition
does not depend on the field $K$.
\begin{remark}
Let $\tilde\tau \in G_K$ be a lift of $\tau \in G(L_{e}/K).$ The coset
$\tilde\tau \, G_{L_{e}}$ does not depend on the lift. The Zariski closure of  
$\rho_{l} (\tilde\tau \, G_{L_{e}}) = \rho_{l} (\tilde\tau) \, \rho_{l} (G_{L_{e}})$
in $Sp_{V_l}$ is  
\begin{equation}
\rho_{l} (\tilde\tau) \, G_{l, L_{e}, 1}^{alg} \subseteq DL_{K}^{\tau}(A)_{\Q_l}.
\end{equation}
\noindent
We observe that
\begin{equation}
G_{l, K, 1}^{alg}
 = \bigsqcup_{\tau \in G(L_{e}/K)} \,\, \rho_{l} (\tilde\tau) \, G_{l, L_{e}, 1}^{alg}.
\label{decomposition into cosets of algebraic closures}\end{equation} 

\noindent
Since $DL_{K}^{id}(A) = DL_{L_{e}}(A) = DL_{\overline{F}}(A)$, we get
\begin{equation}
G_{l, L_{e}, 1}^{alg} \subseteq DL_{L_{e}}(A)_{\Q_l}.
\label{GlLe1alg subset DLLeA}\end{equation} 

\noindent
Hence the following two equalities are equivalent:
\begin{align}
G_{l, L_{e}, 1}^{alg} &= DL_{L_{e}}(A)_{\Q_l},
\label{ItIsSatoTateGroupIdentification Le} \\
G_{l, K, 1}^{alg} &= DL_{K}(A)_{\Q_l}.
\label{ItIsSatoTateGroupIdentification}\end{align}
\end{remark}

\section{Some cases of the algebraic Sato-Tate conjecture}

In general, the containment $H(A) \subseteq L(A)$ can be strict, which makes the Mumford-Tate conjecture a somewhat
subtle problem; for instance, Mumford exhibited examples of simple abelian fourfolds for which $H(A) \neq L(A)$. 
Fortunately, when $A$ has a large endomorphism algebra as compared to its dimension,
such pathologies do not occur: one has $H(A) = L(A)$ and one can often show that the Mumford-Tate conjecture holds.
When this happens, we will say that the Mumford-Tate conjecture for $A$ is \emph{explained by endomorphisms}.

The following theorem asserts that in cases where the Mumford-Tate conjecture is explained by endomorphisms
\emph{and} the twisted decomposable Lefschetz group over $\overline{F}$ is connected, the algebraic Sato-Tate
conjecture is in a sense also explained by endomorphisms.
\begin{theorem}\label{conditions for AST} 
Assume that the following conditions hold.
\begin{itemize} 
\item[1.] We have $ H(A) = L(A)$.
\item[2.] We have $(G_{l, K}^{alg})^{\circ} = MT(A)_{\Q_l}.$
\item[3.] The group $DL_{\overline{F}}(A)$ is connected.
\end{itemize}
Then (\ref{ItIsSatoTateGroupIdentification}) holds for every $l.$ 
Consequently, the algebraic Sato-Tate conjecture 
(Conjecture~\ref{algebraic Sato Tate conj.})
holds for $A/K$ with
\begin{equation} 
AST_K (A) = DL_{K}(A).
\label{SatoTateGroupIdentified}
\end{equation}
\end{theorem}

\begin{proof} 
It is enough to prove (\ref{ItIsSatoTateGroupIdentification Le}).
By our assumptions and Remark \ref{Mumford-Tate Conj and Deligne Theorem} we get
$(G_{l, L_{e}, 1}^{alg})^{\circ} = H (A)_{\Q_l} = L(A)_{\Q_l} = DL_{L_{e}}(A)_{\Q_l}.$ 
Hence  $DL_{L_{e}}(A)_{\Q_l}$ is also connected 
for every $l$ and by (\ref{GlLe1alg subset DLLeA}) we obtain 
$(G_{l, L_{e}, 1}^{alg})^{\circ} = G_{l, L_{e}, 1}^{alg}$ for every $l.$ 
\end{proof}

\begin{remark} Under assumptions of Theorem \ref{conditions for AST} the results of 
Theorems 
\ref{Mumford-Tate group standard properties},  
\ref{Zariski closure standard properties} and  
\ref{Decomposible Lefschetz group standard properties} 
show that the algebraic Sato-Tate conjecture holds for $A^s$ for all $s \geq 1$
with 
$AST_K (A^s) = DL_{K}(A^s) \cong DL_{K}(A) = AST_K (A).$
\end{remark}

Conversely, if the algebraic Sato-Tate conjecture for $A$ is explained by endomorphisms,
so is the Mumford-Tate conjecture.
\begin{theorem}\label{conditions for MT via AST} 
Assume that
(\ref{ItIsSatoTateGroupIdentification}) and (\ref{SatoTateGroupIdentified}) 
hold for every $l$ (so in particular, the algebraic Sato-Tate conjecture holds).
We then have the following.
\begin{itemize} 
\item[1.] We have $H(A) = L(A)$.
\item[2.] We have $(G_{l, K}^{alg})^{\circ} = MT(A)_{\Q_l}$. 
\end{itemize}
\end{theorem}
\begin{proof} 
By Theorem \ref{Theorem Deligne} and Remark  
\ref{Mumford-Tate Conj and Deligne Theorem} (see \eqref{Del ineq1}), we have
\begin{equation}
(G_{l, K, 1}^{alg})^{\circ} \subseteq H(A)_{\Q_l} \subseteq L(A)_{\Q_l} = DL_K(A)_{\Q_l}^{\circ}.
\label{GlK1 connected  and HA and LKA}
\end{equation}  
By (\ref{SatoTateGroupIdentified}) we get
\begin{equation}
(G_{l, K, 1}^{alg})^{\circ} = H(A)_{\Q_l} = L(A)_{\Q_l} = DL_K(A)_{\Q_l}^{\circ}.
\label{GlK1 connected and SatoTate gives HAl equal LKAl}
\end{equation}    
Hence by Remark \ref{Mumford-Tate Conj and Deligne Theorem}, we obtain
$(G_{l, K}^{alg})^{\circ} = MT(A)_{\Q_l}.$
Moreover, since $H(A)$ is closed in $L (A)$, 
(\ref{GlK1 connected and SatoTate gives HAl equal LKAl}) gives
\begin{equation}
H(A) = L(A)
\label{GlK1 connected and SatoTate gives HA equal LKA}
\end{equation}
as desired.
\end{proof}

\begin{remark}
Recall that $L (A) = DL_{K} (A)^{\circ} \triangleleft DL_{K}^{id} (A) 
\triangleleft DL_{K} (A).$ Consider the following epimorphism of groups:
\begin{equation}
DL_{K} (A) / L(A) \rightarrow DL_{K} (A) / DL_{K}^{id}(A) \cong G (L_e / K).
\label{DLKA to Galois Le over K}\end{equation}
If $A$ satisfies the assumptions of Theorem 
\ref{conditions for AST}, then the epimorphism 
(\ref{DLKA to Galois Le over K}) is an isomorphism. 
This gives an identification $G(L_e/K) \cong AST_K(A)/AST_K(A)^\circ$.
\end{remark}

We now assemble some cases where the Mumford-Tate and algebraic Sato-Tate conjecture are explained by endomorphisms.
We start with some cases with $A$ simple.

\begin{definition}
Put $g := \dim(A)$. We say $A$ is of \emph{CM type} if $D$ contains a commutative $\Q$-subalgebra of dimension $2g$.
If $A$ is simple, then $DL_{\overline{F}}(A)$ is a torus of dimension at most $g$; equality holds if $g \leq 3$
\cite[Example~3.7]{R}. In this case, $L(A) = DL_{\overline{F}}(A)$.
\end{definition}

\begin{theorem} \label{CM type}
Let $A/F$ be an abelian variety of CM type. Then
$A$ satisfies the conditions of Theorem~\ref{conditions for AST}, so the algebraic Sato-Tate conjecture holds with $AST_K(A) = DL_K(A)$.
\end{theorem}
\begin{proof}
This was shown by Serre using results of Pohlman; see \cite{Se77}.
\end{proof}

\begin{lemma}\label{BGKCDConn} 
Let $A/F$ be an absolutely simple abelian variety
for which the endomorphism algebra $D$ is of type I, II, or III in the Albert classification.
Then $C_{D} Sp_{(V, \psi)}$ is connected. 
\end{lemma}
\begin{proof} 
This follows from \cite[\S 7]{BGK1} in the type I and II cases and \cite[\S 5]{BGK2} in the type III case.
\end{proof}

\begin{theorem}\label{BGK123} Let $A/F$ be an absolutely simple abelian variety of dimension $g$
for which the endomorphism algebra $D$ is of type I, II, or III in the Albert classification.
Let $E$ be the center of $D$ and put $e := [E:\, \Q]$, $d^2 = [D:E]$.
Assume that $\frac{g}{d e}$ is odd. Then $A$ satisfies the conditions of Theorem~\ref{conditions for AST}, 
so the algebraic Sato-Tate conjecture holds with $AST_K(A) = DL_K(A)$.
\end{theorem} 
\begin{proof}

Recall that $MT (A), \, (G_{l, K,}^{alg})^{\circ}, \,
(G_{l, K, 1}^{alg})^{\circ}$ are independent of the field extension
$K/F.$ We may thus take $K$ large enough that $G_{l,K}^{alg}$ is connected. In this case,
conditions 1 and 2 of Theorem~\ref{conditions for AST} are
established in \cite{BGK1} (in the type I and II case) and \cite{BGK2}
(in the type III case). Condition 3 holds by Lemma~\ref{BGKCDConn}.
\end{proof}

One can also show that the Mumford-Tate conjecture and the algebraic Sate-Tate conjecture 
are explained by endomorphisms
for all abelian varieties of dimension
at most 3. This result is the starting point of the analysis of the Sato-Tate conjecture for abelian
surfaces given in \cite{FKRS}; the corresponding analysis for threefolds has not yet been carried out.
\begin{lemma} \label{reducible cases}
For $i=1,2$, let $A_i/F$ be an abelian variety satisfying the Mumford-Tate conjecture.
Suppose that over $\overline{F}$, $A_1$ has no factors of type IV
while $A_2$ either is of CM type or has no factors of type IV.
Suppose in addition over $\overline{F}$, there is no nontrivial homomorphism from $A_1$ to $A_2$.
Then $A = A_1 \times A_2$ also satisfies the Mumford-Tate conjecture and $H(A) = H(A_1) \times H(A_2)$. 
\end{lemma}
\begin{proof}
Let $G, G_1, G_2$ be the groups $(G_{l,F,1}^{alg})^\circ$ associated to $A, A_1, A_2$.
By hypothesis, we have $G_1 = H(A_1)_{\Q_l}$ and $G_2 = H(A_2)_{\Q_l}$.
Since $A_1$ has no factors of type IV, $G_1$ is semisimple \cite[Lemma~1.4]{Ta}.
Similarly, if $A_2$ has no factors of type IV, then $G_2$ is semisimple; if instead $A_2$
is of CM type, then $G_2$ is a torus.

There is a natural inclusion $G \subseteq G_1 \times G_2$ with the property that the induced projection
maps $\pi_1: G \to G_1$ and $\pi_2: G \to G_2$ are surjective.
Since $G$ is reductive, we can write $G$ as an almost direct
product $G'\cdot T$ with $G'$ semisimple and $T$ a torus.

In case $A_2$ is of CM type, we argue following Imai \cite{I}. In this case, 
$p(G') = p([G,G]) = [p(G), p(G)] = G_1$ because $G_1$ is semisimple, while
$q(T) = q(G) = G_2$ because $G'$ has no nontrivial character.
Hence $\dim(G) \geq \dim(G_1 \times G_2)$, forcing $G = G_1 \times G_2$.
By the same reasoning, $H(A_1) = H(A_1) \times H(A_2)$.

In case $A_2$ has no factors of type IV, then an argument of Hazama \cite[Proposition~1.8]{Ha2}
shows that $H(A) = H(A_1) \times H(A_2)$. The same argument implies that
 if $G \neq G_1 \times G_2$, there is a nonzero $G$-equivariant homomorphism $V_l(A_1) \to V_l(A_2)$. 
However, the latter would  imply the existence of a nonzero homomorphism $A_1 \to A_2$ of abelian varieties
over $\overline{F}$,
by virtue of Faltings's proof of the Tate conjecture for abelian varieties over number fields
\cite{F}. 

In both cases, we conclude $G = H(A)_{\Q_l} = H(A_1)_{\Q_l} \times H(A_2)_{\Q_l}$, as desired.
\end{proof}

\begin{theorem} \label{Algebraic Sato-Tate for surfaces and threefolds}
Let $A/F$ be an abelian variety of dimension at most $3$. Then $A$ satisfies the conditions of Theorem~\ref{conditions for AST}, so the Mumford-Tate conjecture holds with $H(A) = L(A)$, and the algebraic Sato-Tate conjecture holds with $AST_K(A) = DL_K(A)$.
\end{theorem}
\begin{proof}
Since we are free to enlarge $F$, we may assume that all simple factors of $A$
are absolutely simple.
We may invoke \cite[Theorem~0.1]{MZ} to establish condition 1 of Theorem~\ref{conditions for AST} in all cases
we are considering,
so we focus on conditions 2 and 3. 

We start with the case where $A$ is simple.
If $A$ is a CM elliptic curve, it is straightforward to check condition 2 (or one can apply Theorem~\ref{CM type})
and condition 3 holds because $DL_{\overline{F}}(A)$ is a torus.
If $A$ is a non-CM elliptic curve, condition 2 follows from Serre's open image theorem or from Theorem~\ref{BGK123},
and condition 3 holds because $DL_{\overline{F}}(A) = SL_2$.
In the other cases, the dimension is prime, so we may deduce condition 2 from \cite[Concluding Remark]{Ch}
and condition 3 from a theorem of Tankeev \cite[Theorem~2.7]{MZ}.

We next consider the case where $A$ is not simple. 
By Theorem~\ref{Decomposible Lefschetz group standard properties},
condition 3 for $A$ reduces to the corresponding condition for each of the simple factors of $A$. We thus need only check condition 2.
In case $A$ is isogenous to a product of elliptic curves, this is a result of Imai \cite{I}.
The only other possibility is that $A$ is isogenous to the product of an elliptic curve $E$ with a simple
abelian surface $B$; this case may be treated using Lemma~\ref{reducible cases}.
\end{proof}

\begin{remark}
As noted earlier, Theorem~\ref{Algebraic Sato-Tate for surfaces and threefolds} cannot be extended to dimension 4,
due in part to Mumford's examples of absolutely simple abelian fourfolds for which the Hodge conjecture is not
explained by endomorphisms. However, there are also some nonsimple examples with the same property.
For a complete classification of Hodge groups for abelian varieties of dimension at most 5, see \cite{MZ}.
\end{remark}


\section{Category of motives for absolute Hodge cycles}

Even for abelian varieties for which the Mumford-Tate conjecture is not explained by endomorphisms, one can 
infer from Serre \cite{Se94} a proposed construction of the general algebraic Sato-Tate group
in terms of motivic Galois groups and absolute\footnote{Since we consider only the motives associated to abelian varieties, the adjective \emph{absolute} is rendered unnecessary by Deligne's theorem equating Hodge cycles
with absolute Hodge cycles \cite{D1}. However, we prefer to speak of \emph{absolute Hodge cycles} in order
to emphasize the action of Galois on them, which is crucial for our construction of the 
disconnected part of the algebraic Sato-Tate group.} Hodge cycles \cite{D1, D, DM}. In the remainder of the paper,
we make the $\ell$-adic interpretation of this construction explicit, and show that if a suitably
motivic form of the Mumford-Tate conjecture holds for a particular abelian variety, then the algebraic
Sato-Tate conjecture holds as well.

To begin with, we recall some terminology and set some notation concerning Tannakian categories.
Let $\mathcal{C}$ be a unital, $\Q$-linear Tannakian category equipped 
(as in \cite[Definition~2.19]{DM}) with a $\Q$-linear, faithful tensor functor
$\omega\, :\, \mathcal{C} \rightarrow {\rm Vec}_{\Q}$,  called the \emph{fiber functor}
of $\mathcal{C}$.

\begin{definition}
For $X, Y \in {{\rm obj}} (\mathcal{C})$,
let ${\underline{\rm Hom}}_{\mathcal{C}}(X,\, Y) \in {{\rm obj}} (\mathcal{C})$ 
be the object 
which represents the functor 
\begin{equation}
\mathcal{C}^{\circ} \rightarrow {\bf{{\rm Set}}}, \qquad
T \mapsto {\rm Hom}_{\mathcal{C}} (T \otimes X, \,Y).
\nonumber\end{equation}
That is, for every $T \in {{\rm obj}} (\mathcal{C})$ there is a natural bijection
$${\rm Hom}_{\mathcal{C}}(X \otimes T,\, Y) \cong 
{\rm Hom}_{\mathcal{C}}(T, {\underline{\rm Hom}}_{\mathcal{C}}(X,\, Y)).$$  
One puts $X^{\vee} := {\underline{\rm Hom}}_{\mathcal{C}}(X,\, \underline{1})$
and one gets ${\underline{\rm Hom}}_{\mathcal{C}}(X,\, Y) \cong X^{\vee} \otimes Y.$ 
\end{definition}
\medskip

\begin{proposition} {\rm(\cite[Proposition~1.9]{DM})} For  
a functor $F \, :\, \mathcal{C} \rightarrow \mathcal{C}^{\prime}$ of unital Tannakian categories
and $X, Y \in {{\rm obj}} (\mathcal{C})$,
there is a natural isomorphism
\begin{equation}
F ({\underline{\rm Hom}}_{\mathcal{C}}(X,\, Y)) \cong {\underline{\rm Hom}}_{\mathcal{C}^{\prime}}(F(X),\, F(Y)).
\label{F preserves underline Hom}\end{equation}
\end{proposition}

Let us consider now the case of motives. 
\begin{definition}
Let $K$ be a number field and let $\mathcal{M}_{K}$ be the motivic category for absolute Hodge cycles:
${{\rm Hom}}_{\mathcal{M}_{K}} (M, \, N) := \mathrm{Mor}_{AH}^{0} (X, \, Y)$ (see \cite{DM}).
Define $\mathcal{M}_{\overline{K}}$ similarly. Let $H_{B}$ be the fiber functor given by Betti realization:
\begin{equation}
H_{B} : \mathcal{M}_{K} \rightarrow {\rm Vec}_{\Q}.
\label{fiber functor1}\end{equation}
\end{definition}
\medskip

The functor $H_B$ factors through the functor
\begin{equation}
\mathcal{M}_{K} \rightarrow 
\mathcal{M}_{\overline{K}}, \qquad
M \mapsto \overline{M} := M \otimes \overline{K}.
\label{tensor over K with overline K}
\end{equation}

\begin{definition}
Let $\{X_i\}$ be a family of varieties over $K$ and let $\mathcal{C} \subset \mathcal{M}_{K}$ be the full 
Tannakian  subcategory generated by $\{h (X_i)\}.$ Let $\overline{\mathcal{C}}$ be the smallest full Tannakian
subcategory of $\mathcal{M}_{\overline{K}}$ containing the image of $\mathcal{C}$ via 
the functor (\ref{tensor over K with overline K}). 
\end{definition}
\medskip

The category
$\overline{\mathcal{C}}$ is the full Tannakian subcategory of $\mathcal{M}_{\overline{K}}$ generated by
$\{h (\overline{X_i})\},$ c.f. \cite[pp.\ 215-216]{DM}.
For $M, N \in  {{\rm obj}} (\mathcal{C})$
the natural map
\begin{equation}
{{\rm Hom}}_{\overline{\mathcal{C}}} (\overline{M}, \, \overline{N}) \times \overline{M} \rightarrow \overline{N}
\label{def of underline hom 1}\end{equation}
gives the natural map
\begin{equation}
{{\rm Hom}}_{\overline{\mathcal{C}}} (\overline{M}, \, \overline{N}) \rightarrow  
{\underline{\rm Hom}}_{\overline{\mathcal{C}}}(\overline{M},\, \overline{N})
\label{def of underline hom 2}\end{equation}
where ${{\rm Hom}}_{\overline{\mathcal{C}}} (\overline{M}, \, \overline{N}) = 
{{\rm Hom}}_{\mathcal{M}_{\overline{K}}} (\overline{M}, \, \overline{N})$ is a finite  dimensional
vector space which is also a discrete $G_K$-module, so it can be considered as an Artin motive
in $\mathcal{M}_{K}.$ Moreover by (\ref{F preserves underline Hom}) we get 
${\underline{\rm Hom}}_{\overline{\mathcal{C}}}(\overline{M},\, \overline{N})
\cong {\underline{\rm Hom}}_{\mathcal{M}_{\overline{K}}}(\overline{M},\, \overline{N}).$
Hence the map (\ref{def of underline hom 2}) is naturally isomorphic to the map
\begin{equation}
{{\rm Hom}}_{\mathcal{M}_{\overline{K}}} (\overline{M}, \, \overline{N}) \rightarrow  
{\underline{\rm Hom}}_{\mathcal{M}_{\overline{K}}}(\overline{M},\, \overline{N}).
\label{def of underline hom 3}\end{equation}
Let $L/K$ be a field extension with $L \subset \overline{K}.$  For any variety $X/L$ of pure dimension $m$, one sets 
$h(X)^{\vee} := h(X)(m)$, which extends to the contravariant functor
$$\mathcal{M}_L \rightarrow \mathcal{M}_L, \qquad M \mapsto M^{\vee}$$
and one sets (see \cite[p. 205]{DM})
\begin{equation}
{\underline{\rm Hom}}_{\mathcal{M}_L}(M,\, N) := M^{\vee}\otimes N.
\label{explicit computation of underline hom}\end{equation} 
Since the map (\ref{def of underline hom 1}) commutes with the $G_K$-action,
using (\ref{explicit computation of underline hom}) the  
maps (\ref{def of underline hom 2}) and (\ref{def of underline hom 3}) also commute 
with the  $G_K$-action. 

\begin{proposition} \label{imbedding of hom in motivic category}
There is a natural morphism of motives in $\mathcal{M}_{K}$:
\begin{equation}
{{\rm Hom}}_{\mathcal{M}_{\overline{K}}} (\overline{M}, \, \overline{N}) \rightarrow  
{\underline{\rm Hom}}_{\mathcal{C}}(M,\, N),
\label{def of underline hom 4}\end{equation}
which is an embedding of motives.
\end{proposition}
\begin{proof}
The existance of the morphism (\ref{def of underline hom 4}) follows from the definition of morphisms for the category 
of motives for absolute Hodge cycles.
Let $\underline{1}$ be the unit of $\mathcal{M}_{\overline{K}}.$ 
Applying ${{\rm Hom}}_{\mathcal{M}_{\overline{K}}} (\underline{1}, \,\, . \,)$ to the map
(\ref{def of underline hom 3}) one immediately gets the identity map 
${{\rm Hom}}_{\mathcal{M}_{\overline{K}}} (\overline{M}, \, \overline{N}) \rightarrow
{{\rm Hom}}_{\mathcal{M}_{\overline{K}}} (\overline{M}, \, \overline{N}).$ This shows that the map
(\ref{def of underline hom 4}) is an embedding of motives.
\end{proof}


\section{Motivic Mumford-Tate and Sato-Tate groups}

Using the Tannakian formalism, we now define a motivic Sato-Tate group
by reprising the setup from \S \ref{section:Zariski closure}.
\begin{definition}
Let $G_{\mathcal{M}_{K}} := {\rm Aut}^{\otimes} (H_{B})$ and 
let $\mathcal{M}_{K} (A)$ be the smallest Tannakian subcategory of $\mathcal{M}_{K}$ containing $A.$ 
Let $H_{B} | \mathcal{M}_K (A)$ be the restriction of $H_B$ to the subcategory 
$\mathcal{M}_K (A).$ Put $G_{\mathcal{M}_{K} (A)} := {\rm Aut}^{\otimes} (H_{B} | \mathcal{M}_{K} (A)).$ 
\end{definition}
\medskip

The algebraic groups $G_{\mathcal{M}_{K} (A)}$ are reductive
but not necessarily connected (see \cite[p.\ 379]{Se94}). 
Observe that the ring of endomorphisms $D := {\rm End}_{\overline{K}} (A) \otimes_{\Z} \Q$ 
is a discrete $G_K$-module which is also a finite dimensional $\Q$-vector space. Moreover (see \cite[p. 213]{DM}),
\begin{equation}
{\rm End}_{\overline{K}} (A) \otimes_{\Z} \Q = 
{\rm End}_{\overline{K}} (\overline{A}) \otimes_{\Z} \Q = 
{\rm End}_{\mathcal{M}_{\overline{K}}} (h^1(\overline{A})).
\label{end of abel var is the h1 end}\end{equation}

We may thus consider $D$ as an Artin motive, as follows.
Recall that $\mathcal{M}_{K}^{0}$ is equivalent to ${\rm Rep}_{\Q} (G_K)$,
the category of finite dimensional $\Q$-vector spaces with continuous actions of $G_K.$ 

\begin{definition}
Let $h^{0}(D)$ denote the Artin motive corresponding to $D.$ 
Let $\mathcal{M}_{K}^{0} (D)$ be the smallest Tannakian subcategory of 
$\mathcal{M}_{K}^{0}$ containing $h^{0}(D)$ and put 
$G_{\mathcal{M}_{K}^{0} (D)} := {\rm Aut}^{\otimes} (H_{B}^{0} | \mathcal{M}_{K}^{0} (D)).$
\label{def of h0D}\end{definition}
\medskip

\noindent
By (\ref{F preserves underline Hom}), Proposition \ref{imbedding of hom in motivic category}, and 
(\ref{end of abel var is the h1 end}), we have a natural embedding of motives 
\begin{equation}
h^0(D) \subset \underline{{\rm End}}_{\mathcal{M}_{K} (A)} (h^1(A)) = 
\underline{{\rm End}}_{\mathcal{M}_{K}} (h^1(A)).
\label{h0D imbeds into end h1 A}\end{equation}
Note that ${\underline{\rm End}}_{\mathcal{M}_{K} (A)} (h^1 (A)) = 
h^1(A) \otimes h^1(A)^{\vee} \in \mathcal{M}_{K} (A).$ 
Recall that $G_{\mathcal{M}_{K}^{0}} \cong G_K$ so we observe that
\begin{equation}
G_{\mathcal{M}_{K}^{0} (D)} \cong G(L_{e}/K).
\nonumber\end{equation}
Since $\mathcal{M}_{K}$ is semisimple \cite[Prop.\ 6.5]{DM}, the motive $h^0(D) $ splits off of 
$\underline{{\rm End}}_{\mathcal{M}_{K} (A)} (h^1(A))$ in $\mathcal{M}_{K}.$
Moreover the semisimplicity of $\mathcal{M}_{K}$, together with the observation that 
$\mathcal{M}_{K}^{0}$ and $\mathcal{M}_K (A)$ are full subcategories of
$\mathcal{M}_K$, shows that the top horizontal and left vertical maps in the following diagram 
are faithfully flat
(see \cite[(2.29)]{DM}):
\begin{equation}
{\xymatrix{
G_{\mathcal{M}_{K}} \ar@{>>}[d]^{} \ar@{>>}[r]^{}  & G_{K} \ar@{>>}[d]^{} \\
G_{\mathcal{M}_{K} (A)} \ar@{>>}[r]^{}  & G(L_e/K)}}
\label{surjection of motivic Galois groups}
\end{equation}
In particular all homomorphisms in \eqref{surjection of motivic Galois groups} are surjective.
\bigskip

Since $V := H_{1} (A(\C), \, \Q)$ admits a $\Q$-rational polarized 
Hodge structure of weight $1$, we have
\begin{equation}
G_{\mathcal{M}_{K}(A)} \subset GSp_{(V, \psi)}. 
\label{GMKA subset of GSpV}\end{equation}

\begin{definition}
Define the following algebraic groups:
\begin{align*}
G_{\mathcal{M}_{K}(A), 1}  &:= G_{\mathcal{M}_{K}(A)} \cap Sp_V \\
G_{\mathcal{M}_{K}(A), 1}^{\circ} &:= (G_{\mathcal{M}_{K}(A)})^{\circ} \cap Sp_V.
\end{align*}
\end{definition}

\begin{remark}
In \cite[p.\ 396]{Se94} the group $G_{\mathcal{M}_{K}(A), 1}$ is denoted $G_{\mathcal{M}_{K}(A)}^1.$
\end{remark}
\begin{definition}
For any $\tau \in G(L_{e}/K)$, put
\begin{equation}
GSp_{(V, \psi)}^{\tau} := 
\{g \in GSp_{(V, \psi)}:\,\, g \beta g^{-1} = 
\rho_{e}(\tau)(\beta) \,\ \forall \beta \in D\}.
\label{def of GSptau}\end{equation}
\end{definition}
\bigskip

We have the following equality:
\begin{equation}
GSp_{(V, \psi)}  = \bigsqcup_{\tau \in G(L_{e}/K)} \, GSp_{(V, \psi)}^{\tau}.
\label{GSp decomposition into GSptau}
\end{equation}
Observe that
\noindent
\begin{equation}
GSp_{(V, \psi)}^{Id}  = C_{D}( GSp_{(V, \psi)}).
\label{GSpId is CDGSp}
\end{equation}

\begin{remark}
The bottom horizontal arrow in the diagram (\ref{surjection of motivic Galois groups}) is
\begin{equation}
G_{\mathcal{M}_{K} (A)} \rightarrow G_{\mathcal{M}_{K} (D)} \cong G(L_{e}/K).
\label{map from GMKA to Gal}\end{equation}
Let $g \in G_{\mathcal{M}_{K} (A)}$ and let $\tau := \tau (g)$ be the image of $g$ via 
the map (\ref{map from GMKA to Gal}). Hence for any element $\beta \in D$ considered as an 
endomorphism of $V$ we have:
\begin{equation}
g \beta g^{-1} = \rho_{e}(\tau)(\beta).
\label{GMKA1 acts on V via Galois}
\end{equation}
\end{remark}

\begin{definition}
For any $\tau \in G(L_{e}/K)$, put
\begin{equation}
G_{\mathcal{M}_{K}(A)}^{\tau} := 
\{g \in G_{\mathcal{M}_{K}(A)}:\,\, g \beta g^{-1} = 
\rho_{e}(\tau)(\beta), \,\ \forall \beta \in D\}.
\label{def of GMKAtau}\end{equation}
\end{definition}
\bigskip

It follows from (\ref{GMKA1 acts on V via Galois}), (\ref{def of GMKAtau}), and the surjectivity of 
(\ref{map from GMKA to Gal}) that
\begin{equation}
G_{\mathcal{M}_{K}(A)} = \bigsqcup_{\tau \in G(L_{e}/K)} \, G_{\mathcal{M}_{K}(A)}^{\tau}
\label{GMKA decomposition into GMKAtau}
\end{equation}
It is clear from (\ref{GMKA subset of GSpV}) and (\ref{def of GSptau}) that 
\begin{equation}
G_{\mathcal{M}_{K}(A)}^{\tau} \subset GSp_{(V, \psi)}^{\tau}. 
\label{GMKAtau subset of GSptau}\end{equation}
Hence (\ref{GMKA1 acts on V via Galois}) and (\ref{GMKA decomposition into GMKAtau}) give
\begin{equation}
(G_{\mathcal{M}(A)})^{\circ}  \triangleleft G_{\mathcal{M}_{K}(A)}^{{Id}}  \triangleleft  G_{\mathcal{M}_{K}(A)}.
\label{GMKAId normal div in GMKA}
\end{equation}
The map (\ref{map from GMKA to Gal}) gives the following natural map: 
\begin{equation}
G_{\mathcal{M}_{K} (A), 1} \rightarrow  G(L_{e}/K).
\label{map from GMKA1 to Gal}\end{equation}

\begin{definition}
For any $\tau \in G(L_{e}/K)$ put
\begin{equation}
G_{\mathcal{M}_{K}(A), 1}^{\tau} := 
\{g \in G_{\mathcal{M}_{K}(A), 1}:\,\, g \beta g^{-1} = 
\rho_{e}(\tau)(\beta), \,\ \forall \beta \in D\}.
\label{def of GMKA1tau}\end{equation}
\end{definition}
\bigskip

It follows that there is the following equality
\begin{equation}
G_{\mathcal{M}_{K}(A), 1}^{\tau} = G_{\mathcal{M}_{K}(A), 1} \cap G_{\mathcal{M}_{K}(A)}^{\tau}.
\label{GMKA1tau equals GMKA1 cap GMKAtau}
\end{equation}
Hence for $\tau = \tau (g)$, we have
\begin{align}
G_{\mathcal{M}_{K}(A), 1}^{\tau} &\subset DL_{K}^{\tau}(A)
\label{GMKAtau subset DLKAtau} \\
G_{\mathcal{M}_{K}(A), 1} &\subset DL_{K} (A). 
\label{GMKA1 subset of DLKA}
\end{align}
Naturally we also have
\begin{equation}
(G_{\mathcal{M}(A), 1})^{\circ}  \triangleleft G_{\mathcal{M}_{K}(A), 1}^{{Id}}.
\triangleleft  G_{\mathcal{M}_{K}(A), 1}
\label{GMKA1Id normal div in GMKA1}
\end{equation}
Hence by (\ref{GMKA1tau equals GMKA1 cap GMKAtau}) we get
\begin{equation}
G_{\mathcal{M}(A), 1} / G_{\mathcal{M}_{K}(A), 1}^{{Id}} \subset 
G_{\mathcal{M}(A)} / G_{\mathcal{M}_{K}(A)}^{{Id}}.   
\label{GMKA1 mod GMKA1Id subset GMKA mod GMKAId}
\end{equation}
\bigskip

We have the following analogue of Theorem~\ref{equality of conn comp for Glalg and Glalg1}. 
\begin{theorem}\label{equality of conn comp for GM and GM1} 
Assume that $G_{\mathcal{M}_{K}(A), 1}^{\circ}$ is connected. Then the following map is an isomorphism:
$$i_{M} \,\, : \,\,  G_{\mathcal{M}_{K}(A), 1}/(G_{\mathcal{M}_{K}(A), 1})^{\circ} \,\, 
{\stackrel{\cong}{\longrightarrow}} \,\, G_{\mathcal{M}_{K}(A)} /(G_{\mathcal{M}_{K}(A)})^{\circ}.$$  
\end{theorem}
\begin{proof}  
Consider the following diagram, in which we write $\mathcal{M}(A)$ for $\mathcal{M}_{K} (A)$
to make the notation simpler.  
$$
\xymatrix{
\quad & \quad 1 \ar@<0.1ex>[d]^{}  \quad & \quad 1 \ar@<0.1ex>[d]^{} \quad & 
1 \ar@<0.1ex>[d]^{} \\ 
1 \ar[r]^{} \quad & \quad  (G_{\mathcal{M}(A), 1})^{\circ}\ar@<0.1ex>[d]^{} \ar[r]^{}  \quad 
& \quad G_{\mathcal{M}(A), 1} \ar@<0.1ex>[d]^{} \ar[r]^{}  \quad 
& \quad \pi_{0} (G_{\mathcal{M}(A), 1})   
\ar@<0.1ex>[d]^{i_{M}} \ar[r]^{}  \quad 
& \,\, 1 \\ 
1 \ar[r]^{} \quad & \quad (G_{\mathcal{M}(A)})^{\circ} \ar@<0.1ex>[d]^{} \ar[r]^{}  \quad 
& \quad  G_{\mathcal{M}(A)} \ar@<0.1ex>[d]^{} \ar[r]^{}  \quad 
& \quad \pi_{0}(G_{\mathcal{M}(A)})   \ar@<0.1ex>[d]^{} \ar[r]^{}  \quad 
&  \,\, 1 \\
1 \ar[r]^{} \quad & \quad \G_m  \ar@<0.1ex>[d]^{}  \ar[r]^{=}  
\quad & \quad \G_m  \ar@<0.1ex>[d]^{}  \ar[r]^{} \quad & 1\\
 \quad & \quad 1   \quad & 1 \\}
\label{motivic diagram in Serre theorem}$$

By definition the rows are exact and the middle column is exact. 
Hence the map $i_M$ is surjective. Since $G_{\mathcal{M}_{K}(A), 1}^{\circ}$ has the 
same dimension as $G_{\mathcal{M}_{K}(A), 1}$ and by assumption 
$G_{\mathcal{M}_{K}(A), 1}^{\circ}$ is connected,
we then have $G_{\mathcal{M}_{K}(A), 1}^{\circ} = (G_{\mathcal{M}_{K}(A), 1})^{\circ}.$ 
Hence the left column is also exact. This shows that $i_M$ is an isomorphism.
\end{proof}

\begin{remark}
Since $G_{\mathcal{M}(A)}$ is reductive, the middle vertical column of the diagram 
of the proof of Theorem~\ref{equality of conn comp for GM and GM1}) 
shows that $G_{\mathcal{M}(A), 1}$ is also reductive.
\label{GMKA1 is reductive}\end{remark}

\begin{corollary}\label{equality of quotients concerning for GM and GM1} 
Assume that $G_{\mathcal{M}_{K}(A), 1}^{\circ}$ is connected. Then there are natural isomorphisms
\begin{align}
G_{\mathcal{M}_{K}(A), 1}/G_{\mathcal{M}_{K}(A), 1}^{Id} \,\, 
& {\stackrel{\cong}{\longrightarrow}} \,\, G_{\mathcal{M}_{K}(A)} /G_{\mathcal{M}_{K}(A)}^{Id},
\label{GMKA1 mod GMKA1Id equals GMKA mod GMKAId} \\
G_{\mathcal{M}_{K}(A), 1}^{Id}/(G_{\mathcal{M}_{K}(A), 1})^{\circ} \,\, 
&{\stackrel{\cong}{\longrightarrow}} \,\, G_{\mathcal{M}_{K}(A)}^{Id} /(G_{\mathcal{M}_{K}(A)})^{\circ}, 
\label{GMKA1Id mod GMKA1circ equals GMKAId mod GMKAcirc} \\
G_{\mathcal{M}_{K}(A), 1}/G_{\mathcal{M}_{K}(A), 1}^{Id} \,\, 
&{\stackrel{\cong}{\longrightarrow}} \,\, DL_{K} (A) / DL_{K} (A)^{Id} 
\,\, {\stackrel{\cong}{\longrightarrow}} \,\, G(L_e/K). 
\label{GMKA1 mod GMKA1Id equals DLKA mod DLKAId equals GLeK}
\end{align}
In particular the natural map (\ref{map from GMKA1 to Gal}) is surjective.
\end{corollary}
\begin{proof}  This follows from (\ref{GMKAId normal div in GMKA}), 
(\ref{GMKA1Id normal div in GMKA1}), (\ref{GMKA1 mod GMKA1Id subset GMKA mod GMKAId}),
the surjectivity of (\ref{map from GMKA to Gal}), and Theorem 
\ref{equality of conn comp for GM and GM1}.
\end{proof}

\begin{remark}
The $l$-adic representation $\rho_{l} \, : \, G_K \rightarrow G_{\mathcal{M}_{K}(A)} (\Q_l)$
associated with $A$ factors through $G_{\mathcal{M}_{K}(A)} (\Q_l)$
(see \cite[p.\ 386]{Se94}).
Hence
\begin{equation}
G_{l, K}^{alg} \subset {G_{\mathcal{M}_{K}(A)}}_{\Q_l} 
\label{GlKalg subset GMKAQl}\end{equation}
where ${G_{\mathcal{M}_{K}(A)}}_{\Q_l} := {G_{\mathcal{M}_{K}(A)}} \otimes_{\Q} \Q_l.$
By the commutativity of the diagram
$$
\xymatrix{
\quad G_{l, K, 1}^{alg} \ar@<0.1ex>[d]^{} \ar[r]^{}   \quad & 
\quad {G_{\mathcal{M}_{K}(A), 1}}_{\Q_l}  \ar@<0.1ex>[d]^{} \\ 
\quad G_{l, K}^{alg} \ar@<0.1ex>[d]^{} \ar[r]^{}   \quad & 
\quad {G_{\mathcal{M}_{K}(A)}}_{\Q_l} \ar@<0.1ex>[d]^{} \\
\quad\G_{m}  \ar[r]^{=}  \quad & \quad  \G_{m} \\
}\label{GlK1 subset GMKA1Ql}$$ 
it follows that
\begin{equation}
G_{l, K, 1}^{alg} \subset {G_{\mathcal{M}_{K}(A), 1}}_{\Q_l}. 
\label{GlK1alg subset GMKA1Ql}\end{equation}
\end{remark}

\begin{definition}
The algebraic groups
\begin{align*}
MMT_K (A) & := G_{\mathcal{M}_{K}(A)} \\
MST_K (A) & := G_{\mathcal{M}_{K}(A), 1}
\end{align*}
will be called the \emph{motivic Mumford-Tate group} and \emph{motivic Sato-Tate group} for $A.$
\label{def of motivic mumford Tate and Sato Tate groups}\end{definition} 

\begin{conjecture} (Motivic Mumford-Tate) \label{Motivic Mumford-Tate} 
For any prime number $l$,
\begin{equation}
G_{l, K}^{alg} = {MMT_K (A)}_{\Q_l} .
\label{MMT eq}\end{equation}
\end{conjecture}

By the diagram above, Conjecture \ref{Motivic Mumford-Tate} is equivalent to the following.
\begin{conjecture} (Motivic Sato-Tate) \label{Motivic Sato-Tate} 
For any prime number $l$,
\begin{equation}
G_{l, K, 1}^{alg} = MST_{K}(A)_{\Q_l} .
\label{MST eq}\end{equation}
\end{conjecture}

\begin{remark}
Conjecture \ref{Motivic Mumford-Tate} is equivalent to the conjunction of the following  
equalities:
\begin{align}
(G_{l, K}^{alg})^{\circ} &= ({MMT_K (A)}_{\Q_l})^{\circ} 
\label{MMT0 eq} \\
\pi_{0} (G_{l, K}^{alg}) &= \pi_{0} ({MMT_K (A)}_{\Q_l}).
\label{GMT0C eq}
\end{align}
Similarly, Conjecture \ref{Motivic Sato-Tate} is equivalent to the conjunction of the following equalities:
\begin{align}
(G_{l, K, 1}^{alg})^{\circ} &= ({MST_K (A)}_{\Q_l})^{\circ} 
\label{MST0 eq} \\
\pi_{0} (G_{l, K, 1}^{alg}) &= \pi_{0} ({MST_K (A)}_{\Q_l}).
\label{MST0C eq}
\end{align}
\end{remark}

\begin{theorem}\label{Motivic Galois groups standard properties} 
We have the following results.
\begin{itemize} 
\item[1.] An isogeny $\phi : A_1 \rightarrow A_2$ of abelian varieties over a number field $K$ induces isomorphisms 
$G_{\mathcal{M}_{K}(A_1)}\cong G_{\mathcal{M}_{K}(A_2)}$ and 
$G_{\mathcal{M}_{K}(A_1), 1}\cong G_{\mathcal{M}_{K}(A_2), 1}$.

\item[2.] If $A$ is an abelian variety over a number field $K$, then
for every positive integer $s,$ \,\,\,
$G_{\mathcal{M}_{K}(A^s)}\cong G_{\mathcal{M}_{K}(A)}$ and 
$G_{\mathcal{M}_{K}(A^s), 1}\cong G_{\mathcal{M}_{K}(A), 1}$.
\end{itemize}
\end{theorem}

\begin{proof}
1. Observe that an isogeny $\phi$ induces a natural isomorphism
$V (A_1) \cong V (A_2)$ and a natural equivalence of categories
$\mathcal{M}_{K}(A_1) \cong \mathcal{M}_{K}(A_2).$ This leads to 
a natural equivalence of fiber functors 
$H_{B} \, | \, \mathcal{M}_{K}(A_1) \cong H_{B} \, | \, \mathcal{M}_{K}(A_2).$
\medskip

\noindent
2. By the K{\" u}nneth formula, there is an isomorphism of motives 
$h^1 (A^s) \cong h^1 (A)^s.$ In addition for any abelian variety $B$ over $K$
we have $h(B) = \wedge h^1 (B).$ This shows that the diagonal map
$\Delta \,:\, A \rightarrow A^s$ induces a natural equivalence of categories
$\mathcal{M}_{K}(A) \cong \mathcal{M}_{K}(A^s).$ 
Moreover there is a natural isomorphism $V(A^s) \cong V (A)^s.$ Hence there is a 
natural equivalence of fiber functors 
$H_{B} \, | \, \mathcal{M}_{K}(A^s) \cong H_{B} \, | \, \mathcal{M}_{K}(A).$ 
\end{proof}

\begin{corollary}\label{MMT and MST true for A implies true for As}
Let $A/K$ be an abelian variety such that the motivic Mumford-Tate and 
motivic Sato-Tate conjectures hold for $A.$ Then the motivic Mumford-Tate 
and motivic Sato-Tate conjectures hold for $A^s$ for every positive integer $s.$
\end{corollary}
\begin{proof} This follows from Theorems \ref{Zariski closure standard properties}
and \ref{Motivic Galois groups standard properties}. 
\end{proof}


\section{The algebraic Sato-Tate group}

To conclude, we introduce a candidate for the algebraic Sato-Tate group for an arbitrary abelian variety.

\begin{remark}
One observes (cf.\ \cite[p.\ 379]{Se94}) that 
\begin{equation}
MT (A) \subset (G_{\mathcal{M}_{K}(A)})^{\circ}
\label{MT subset of conn comp of MMT}\end{equation}
and it was proven by Deligne that
\begin{equation}
(G_{l, K}^{alg})^{\circ} \subset MT (A)_{\Q_l}.
\label{Deligne Theorem}
\end{equation}
It follows by (\ref{GMKAtau subset DLKAtau}),
(\ref{GMKA1Id normal div in GMKA1}), and (\ref{MT subset of conn comp of MMT}) that
for every abelian variety $A/K$ there is a sequence of natural containments
\begin{equation}
H (A) \subset  (G_{\mathcal{M}_{K}(A), 1})^{\circ} \subset 
G_{\mathcal{M}_{K}(A), 1}^{Id} \subset C_{D} (Sp_{(V, \psi)}). 
\label{HA subset GMKA10 subset CDSp}
\end{equation}
In the same way, it follows by (\ref{GSpId is CDGSp}), 
(\ref{GMKAtau subset of GSptau}), (\ref{GMKAId normal div in GMKA}), and (\ref{MT subset of conn comp of MMT}) that
for every abelian variety $A/K$ there is a sequence of natural containments
\begin{equation}
MT (A) \subset (G_{\mathcal{M}_{K}(A)})^{\circ} \subset G_{\mathcal{M}_{K}(A)}^{Id} \subset C_{D} (GSp_{(V, \psi)}). 
\label{MTA subset GMKA0 subset CDGSp}
\end{equation}
We observe that the equality
\begin{equation}
H (A) = C_D(Sp_{(V, \psi)})
\label{HA = CD Sp for BGK classes}\end{equation}
is equivalent to the following equality:
\begin{equation}
MT (A) = C_D(GSp_{(V, \psi)}).
\label{MTA = CD GSp for BGK classes}\end{equation}
\end{remark}

\begin{remark}
It is conjectured \cite[sec.\ 3.4]{Se94} that $MT (A) = MMT_{K} (A)^{\circ} = (G_{\mathcal{M}_{K}(A)})^{\circ}$, and 
in \cite[p.\ 380]{Se94} J.-P. Serre gave examples where this conjecture holds. 
Theorems~\ref{MTA is GMKA0 is CDGSp} and \ref{MTA is GMKA0 is CDGSp for ab var of dim at most 3}
below gives more evidence for this conjecture.
\end{remark}

\begin{theorem}\label{MTA is GMKA0 is CDGSp}
Let $A/F$ be an absolutely simple abelian variety of dimension $g$
for which the endomorphism algebra $D$ is of type I, II, or III in the Albert classification.
Let $E$ be the center of $D$ and put $e := [E:\, \Q]$, $d^2 = [D:E]$.
Assume that $\frac{g}{de}$ is odd. Then 
\begin{align}
H(A) &= MST_K (A)^{\circ} = (G_{\mathcal{M}_{K}(A), 1})^{\circ} = C_{D} (Sp_{(V, \psi)}), 
\label{equality HA is GMKA10 is CDSp} \\
MT (A) &= MMT_K (A)^{\circ} = (G_{\mathcal{M}_{K}(A)})^{\circ} = C_{D} (GSp_{(V, \psi)}).
\label{equality MTA is GMKA0 is CDGSp}
\end{align}
\end{theorem}
\begin{proof}
By \cite[Cor. 7.19]{BGK1} and \cite[Cor. 5.19]{BGK2}, for abelian varieties satisfying the assumptions of the 
theorem, the equality (\ref{HA = CD Sp for BGK classes}) follows.
\end{proof}

\begin{theorem}\label{MTA is GMKA0 is CDGSp for ab var of dim at most 3}
Let $A/F$ be an abelian variety of dimension at most $3.$ Then the equalities
(\ref{equality HA is GMKA10 is CDSp}) and (\ref{equality MTA is GMKA0 is CDGSp})
hold also for $A.$
\end{theorem}
\begin{proof}
For abelian varieties of dimension at most $3$ the equality (\ref{HA = CD Sp for BGK classes}) 
holds by Theorem \ref{Algebraic Sato-Tate for surfaces and threefolds}.
\end{proof}

\begin{remark}
To obtain $G_{l, K}^{alg}$ as an extension of scalars to $\Q_l$ of an expected algebraic Sato-Tate 
group defined over $\Q,$ the assumption in the following definition is natural in view of 
(\ref{MT subset of conn comp of MMT}), (\ref{Deligne Theorem}),  and 
Theorem \ref{equality of conn comp for GM and GM1}.  
\end{remark}

\begin{definition} \label{GMKA1 as AST group}
Assume that $MT (A) = MMT_{K} (A)^{\circ}.$
Then the \emph{algebraic Sato-Tate group} $AST_K (A)$ is defined as follows:
\begin{equation}
AST_K (A) := MST_K (A) = G_{\mathcal{M}_{K}(A), 1}. 
\label{AST group}
\end{equation}
Every maximal compact subgroup of $AST_{K} (A)(\C)$ will be called
a \emph{Sato-Tate group} associated with $A/K$ and denoted $ST_K(A).$
\end{definition}
\medskip

\begin{corollary}\label{ AST true for A implies true for As}
Let $A/K$ be an abelian variety such that $MT (A) = MMT_{K} (A)^{\circ}$ and
the algebraic Sato-Tate conjecture holds for $A.$ Then the algebraic Sato-Tate conjecture 
holds for $A^s$ for every positive integer $s.$
\end{corollary}
\begin{proof} It follows immediately from Corollary 
\ref{MMT and MST true for A implies true for As}.
\end{proof}

We now prove that the group $AST_{K} (A)$ satisfies the expected
properties stated in Remark \ref{expected properties of the alg Sato-Tate group}.

\begin{theorem}\label{AST as expected extension of HA} Assume that 
$MT (A) = MMT_{K} (A)^{\circ}.$
Then the group $AST_{K} (A)$ is reductive and: 
\begin{align}
AST_{K} (A) &\subset DL_{K}(A),
\label{ASTKA subset DLKA} \\
AST_{K} (A)^{0} &= H(A),
\label{(AST)0} \\
\pi_{0} (AST_{K} (A)) &= \pi_{0} ( MMT_K (A)),
\label{pi0 AST} \\
\pi_{0} (AST_{K} (A)) &= \pi_{0} (ST_{K} (A)).
\label{pi0 AST is pi0 ST}
\end{align}
In addition, there are the following exact sequences:
\begin{gather}
0 \rightarrow H(A) \rightarrow AST_{K} (A) \rightarrow \pi_{0} (AST_{K} (A))
\rightarrow 0
\label{AST group as an extension} \\
0 \rightarrow \pi_{0}( G_{\mathcal{M}_{K}(A), 1}^{Id}) \rightarrow \pi_{0}( AST_{K} (A) )
\rightarrow G(L_e / K) \rightarrow 0.
\label{pi 0 AST group as an extension}
\end{gather}
\end{theorem}

\begin{proof}
The group $AST_{K} (A)$ is reductive by Remark \ref{GMKA1 is reductive}, and  
(\ref{ASTKA subset DLKA}) follows by (\ref{GMKA1 subset of DLKA}). Equality
(\ref{pi0 AST is pi0 ST}) follows since $AST_{K} (A)^{\circ} (\C)$ is a connected 
complex Lie group and any maximal compact subgroup
of a connected complex Lie group is a connected real Lie group. Assuming that 
$MT (A) = (G_{\mathcal{M}_{K}(A)})^{\circ}$, we have
\begin{equation}
G_{\mathcal{M}_{K}(A), 1}^{\circ} = (G_{\mathcal{M}_{K}(A)})^{\circ} \cap Sp_{(V, \psi)} =
MT (A) \cap Sp_{(V, \psi)} = H(A),
\label{GMKA1 is Hodge under natural assumptions}
\end{equation}
so $G_{\mathcal{M}_{K}(A), 1}^{\circ}$ is connected. Hence
\begin{equation}
AST_{K} (A)^{0} = (G_{\mathcal{M}_{K}(A), 1})^{\circ} = G_{\mathcal{M}_{K}(A), 1}^{\circ},
\label{AST GMKAcirc 1 GMKA1 circ}
\end{equation}
so (\ref{(AST)0}) follows.
Hence the equality (\ref{pi0 AST}) and the exact sequence (\ref{AST group as an extension}) may be extracted from the 
top row of the diagram from the 
proof of Theorem \ref{equality of conn comp for GM and GM1}. 
The exactness of the sequence (\ref{pi 0 AST group as an extension}) follows immediately from
Corollary \ref{equality of quotients concerning for GM and GM1}. 
\end{proof}

\begin{remark} The containments (\ref{HA subset GMKA10 subset CDSp}) give some approximation for 
$\pi_{0}(G_{\mathcal{M}_{K}(A), 1}^{Id})$. Moreover, under the assumption 
$MT (A) = (G_{\mathcal{M}_{K}(A)})^{\circ}$ of Definition \ref{GMKA1 as AST group}, we have the following commutative diagrams:

\begin{equation}
\xymatrix{
\,\, 0 \ar[r]^{} \,\, & \,\, H(A) \ar@<0.1ex>[d]^{} \ar[r]^{}   \,\, & 
\,\, AST_{K} (A)  \ar@<0.1ex>[d]^{}  \ar[r]^{} \,\, & \,\, \pi_{0} \, AST_{K} (A)  
\ar@<0.1ex>[d]^{}  \ar[r]^{}   \,\, & \,\, 0\\ 
\,\,  0 \ar[r]^{} \,\, & \,\, L(A)  \ar[r]^{}   \,\, & 
\,\, DL_{K} (A) \ar[r]^{}   \,\, & \,\, \pi_{0} \, DL_{K} (A)  \ar[r]^{}   \,\, & 
\,\, 0\\
}\label{diagram competibility of AST GMKA1 with DLKA 1} 
\end{equation}
\begin{equation}
\xymatrix{
\,\, 0 \ar[r]^{} \,\, & \,\, \pi_{0} \, G_{\mathcal{M}_{K}(A), 1}^{Id} \ar@<0.1ex>[d]^{} \ar[r]^{}   \,\, & 
\,\, \pi_{0} \, AST_{K} (A)  \ar@<0.1ex>[d]^{}  \ar[r]^{} \,\, & \,\, G(L_e / K)  
\ar@<0.1ex>[d]^{=}  \ar[r]^{}   \,\, & \,\, 0\\ 
\,\,  0 \ar[r]^{} \,\, & \,\, \pi_{0} DL_{K}^{Id} (A)  \ar[r]^{}   \,\, & 
\,\, \pi_{0}\, DL_{K} (A) \ar[r]^{}   \,\, & \,\, G(L_e / K)  \ar[r]^{}   \,\, & 
\,\, 0\\
}\label{diagram competibility of AST GMKA1 with DLKA 2}
\end{equation}

\label{approx for pi 0 GMKA1Id}\end{remark}

\begin{corollary}\label{The natural candidate for AST group example 1}
Asume that $H(A) = C_{D} (Sp_{(V, \psi)}).$ Then
\begin{equation}
AST_{K}(A) = DL_{K} (A).
\end{equation}
\end{corollary}
\begin{proof} It follows by the assumption and 
(\ref{HA subset GMKA10 subset CDSp}) that 
$\pi_{0}( G_{\mathcal{M}_{K}(A), 1}^{Id}) = \pi_{0}(DL_{K}^{Id} (A)) = 1.$
Hence the middle vertical arrow in the diagram 
(\ref{diagram competibility of AST GMKA1 with DLKA 2}), which is the right
vertical arrow in the diagram 
(\ref{diagram competibility of AST GMKA1 with DLKA 1}), is an isomorphism.
Since $L(A) := (C_{D} Sp_{(V, \psi)})^{\circ}$, by our assumption we have
$H(A) = L(A)$. Hence the left 
vertical arrow in the diagram (\ref{diagram competibility of AST GMKA1 with DLKA 1}) is 
an isomorphism, so the middle vertical arrow in the diagram 
(\ref{diagram competibility of AST GMKA1 with DLKA 1}) is an isomorphism. 
\end{proof}

\begin{corollary} \label{cor alg Sato-Tate for MT explained by endo}
If $H(A) = C_{D} (Sp_{(V, \psi)})$ and the Mumford-Tate conjecture holds for $A$
then the algebraic Sato-Tate conjecture holds for
$AST_{K}(A) := G_{\mathcal{M}_{K}(A), 1},$ namely
\begin{equation}
G_{l, K, 1}^{alg}  = AST_{K}(A)_{\Q_l}.
\nonumber
\end{equation}
\end{corollary}
\begin{proof} 
By (\ref{GlK1alg subset GMKA1Ql}) and Corollary \ref{The natural candidate for AST group example 1},
we get 
\begin{equation}
G_{l, K, 1}^{alg} \subset AST_{K}(A)_{\Q_l} = DL_{K} (A)_{\Q_l}.
\nonumber\end{equation}
By the assumption $H(A) = DL_{L_e} (A)$, the equivalence of (\ref{ItIsSatoTateGroupIdentification Le}) and
(\ref{ItIsSatoTateGroupIdentification}) show that we need to prove that
$(G_{l, K, 1}^{alg})^{\circ} = H(A)_{\Q_l}$. However, this is equivalent to the 
Mumford-Tate conjecture by Remark \ref{Mumford-Tate Conj and Deligne Theorem}. 
\end{proof}

\begin{remark} 
The results of this and the previous chapter, in particular Theorem 
\ref{AST as expected extension of HA} and Corollaries \ref{The natural candidate for AST group example 1} and
\ref{cor alg Sato-Tate for MT explained by endo}, show that
the group $AST_{K} (A)$ from Definition \ref{GMKA1 as AST group} is a natural candidate for the 
algebraic Sato-Tate group for the abelian variety $A$. 
\end{remark}

\begin{remark}
For absolutely simple $A/F$ for which the endomorphism algebra $D$ is of type I, II, or III 
in the Albert classification with $\frac{g}{de}$ odd and for $A/F$ of dimension at most 3 we 
explicitely determined the motivic Mumford-Tate and motivic Sato-Tate groups (see 
(\ref{pi0 AST}) of Theorem \ref{AST as expected extension of HA}, Corollary
\ref{The natural candidate for AST group example 1}, Theorems~\ref{MTA is GMKA0 is CDGSp} and \ref{MTA is GMKA0 is CDGSp for ab var of dim at most 3}). 
\label{explicit deter. of mot Mumford-Tate and Sato-Tate groups for classes of ab. var.}
\end{remark} 



\end{document}